\documentclass[11pt,a4paper]{amsart}
\usepackage{amssymb,amsmath}
\usepackage{mathpazo}

\headsep=1cm \numberwithin{equation}{section}
\newtheorem{theorem}{Theorem}[section]
\newtheorem{lemma}[theorem]{Lemma}
\newtheorem{proposition}[theorem]{Proposition}
\newtheorem{corollary}[theorem]{Corollary}

\theoremstyle{definition}
\newtheorem{definition}[theorem]{Definition}
\theoremstyle{remark}
\newtheorem{remark}[theorem]{Remark}
\newtheorem{example}[theorem]{Example}
\newtheorem{question}[theorem]{Question}

\newtheorem{acknowledgement}{Acknowledgement}

\newcommand{\Ass}{\operatorname{Ass}}

\newcommand{\Spec}{\operatorname{Spec}}

\newcommand{\rad}{\operatorname{rad}}
\newcommand{\cd}{\operatorname{cd}}

\newcommand{\Ht}{\operatorname{ht}}
\newcommand{\pd}{\operatorname{pd}}
\newcommand{\wdim}{\operatorname{wdim}}
\newcommand{\fd}{\operatorname{fd}}
\newcommand{\gd}{\operatorname{gl dim}}

\newcommand{\V}{\operatorname{V}}
\newcommand{\Wdim}{\operatorname{w.dim}}
\newcommand{\pgrade}{\operatorname{p.grade}}

\newcommand{\cgrade}{\operatorname{c.grade}}

\newcommand{\injdim}{\operatorname{id}}
\newcommand{\Ext}{\operatorname{Ext}}

\newcommand{\Tor}{\operatorname{Tor}}
\newcommand{\Hom}{\operatorname{Hom}}

\newcommand{\Char}{\operatorname{char}}

\newcommand{\depth}{\operatorname{depth}}

\newcommand{\rank}{\operatorname{rank}}

\newcommand{\Max}{\operatorname{Max}}

\newcommand{\vpl}{\operatornamewithlimits{\varprojlim}}

\newcommand{\lo}{\longrightarrow}
\newcommand{\fm}{\frak{m}}
\newcommand{\fp}{\frak{p}}
\newcommand{\fq}{\frak{q}}
\newcommand{\fa}{\frak{a}}
\newcommand{\fb}{\frak{b}}

\begin{document}

\author[Asgharzadeh]{Mohsen Asgharzadeh}
\title[Homological properties  of the perfect and  absolute integral closures ...]
{Homological properties  of the perfect and absolute integral
closures of Noetherian domains}
\address{M. Asgharzadeh, Department of Mathematics, Shahid Beheshti
University, Tehran, Iran-and-School of Mathematics, Institute for
Research in Fundamental Sciences (IPM), P.O. Box 19395-5746, Tehran,
Iran.} \email{asgharzadeh@ipm.ir}



\dedicatory{Dedicated to Paul Roberts}

\begin{abstract}  For a Noetherian local domain $R$
let $R^+$  be the absolute integral closure  of $R$ and let
$R_{\infty}$ be the perfect closure of $R$, when $R$ has prime
characteristic. In this paper we investigate the projective
dimension of residue rings of certain ideals of $R^+$ and
$R_{\infty}$. In particular, we show that any prime ideal of
$R_{\infty}$ has a bounded free resolution of countably generated
free $R_{\infty}$-modules. Also, we show that the analogue of this
result is true for the maximal ideals of $R^+$, when $R$ has residue
prime characteristic. We compute global dimensions of $R^+$ and
$R_{\infty}$ in some cases. Some applications of these results are
given.
\end{abstract}

\maketitle

\section{Introduction}

For an arbitrary domain $R$, the absolute integral closure $R^+$ is
defined  as the integral closure of $R$ inside an algebraic closure
of the field of fractions of $R$. The notion of absolute integral
closure was first studied by Artin in \cite{Ar}, where among other
things, he proved  $R^+$ has only one maximal ideal $\fm_{R^+}$,
when $(R,\fm)$ is local and Henselian. In that paper Artin proved
that the sum of any collection of prime ideals of $R^+$ is either a
prime ideal or else equal to $R^+$.

As were shown by the works \cite{H1}, \cite{He1}, \cite{HH1},
\cite{HL} and \cite{Sm}, the absolute integral closure of Noetherian
domains has a great significance in commutative algebra. One
important result of Hochster and Huneke \cite[Theorem 5.5]{HH1}
asserts that $R^+$ is a balanced big Cohen-Macaulay algebra in the
case $R$ is an excellent local domain of prime characteristic, i.e.,
every system of parameters in R is a regular sequence on $R^+$. As
shown by Hochster, the flat dimension of the residue field of $R^+$
is bounded by dimension of $R$, when $R$ is a Henselian local domain
and has residue prime characteristic, see \cite[Proposition
2.15]{H1}. Furthermore,  the bound is achieved by all mixed
characteristic local domains if and only if the Direct Summand
Conjecture holds in mixed characteristic, see \cite[Theorem
3.5]{AH}. These observations along with \cite{RSS} motivate us to
investigate the projective dimension of residue rings of certain
ideals of the absolute integral closure and the perfect closure of
Noetherian domains. One of our main results is:

\begin{theorem}  \label{main} Let $(R,\fm)$ be a Noetherian
local Henselian domain of residue prime characteristic.
\begin{enumerate}
\item[i)] The $R^+$-module $R^+/ \fm_{R^+}$ has a free resolution of countably generated
free $R^+$-modules of length bounded by $2\dim R$. In particular,
$\pd_{R^+}(R^+/ \fm_{R^+})\leq2\dim R$.
\item[ii)] If  $R$ is one-dimensional  and complete,  then $\gd
R^+=\pd_{R^+}(R^+/ \fm_{R^+})=2$.
\end{enumerate}
\end{theorem}

Let $R$ be a ring of prime characteristic $p$. Recall that the
perfect closure of $R$ is  defined by $$R_{\infty}:=\{x\in
R^+:x^{p^n}\in R \textmd{ for some }n\in\mathbb{N}\cup\{0\}\}.$$ The
analogue of Theorem \ref{main} for $R_{\infty}$ is:

\begin{theorem}  \label{perfect}
Let $(R,\fm)$ be a  Noetherian  local domain  of prime
characteristic.
\begin{enumerate}
\item[i)] Any prime ideal of $R_{\infty}$ has a bounded free resolution of countably generated
free $R_{\infty}$-modules.
\item[ii)] If $\dim R<3$, then $\sup\{\pd_{R_{\infty}}(R_{\infty}/\fp):
\fp\in\Spec(R_{\infty})\}<\infty$.
\item[iii)] If $R_{\infty}$ is coherent (this holds if $R$ is regular), then
$\gd(R_{\infty})\leq\dim R+1$.
\item[iv)]   If $R$ is regular and of dimension one, then $\gd(R_{\infty})=2$.
\end{enumerate}
\end{theorem}

Throughout this paper all rings are commutative, with identity, and
all modules are unital. The organization of this paper is as
follows:

In Section 2, we summarize some results  concerning projective and
flat dimensions of modules over non-Noetherian rings. Also, for the
convenience of reader, we collect some known results about $R^+$ and
$R_{\infty}$ which will be used throughout this work. The
\textit{Local Global Principle Theorem} \cite[Theorem 2.2.7]{BH}
indicates that a Noetherian local ring has finite global dimension
if and only if its residue field has finite projective dimension.
The main result of Section 3 gives an analogue of this theorem for a
certain class of non-Noetherian rings. Sections 4 and 5 are devoted
to the proof of Theorem \ref{main}. In Section 6 by applying Theorem
\ref{main}, we examine the cohomological dimension and the
cohomological depth of almost zero functors that are naturally
associated to the different classes of almost zero modules. We close
Section 6 by an application of almost zero modules, see Proposition
6.8 below. Section 7 is concerned with the proof of Theorem
\ref{perfect}. In Section 8, as another application of Theorem
\ref{main}, we show that some known results for Noetherian rings do
not hold  for general commutative rings. We include some questions
in a final section.

\section{Preliminaries}
In this section we set notation and discuss some facts which will be
used throughout the paper. In the first subsection we collect some
results concerning projective and flat dimensions of modules over
non-Noetherian rings. Also, we recall some known results concerning
the perfect and absolute integral closures of Noetherian domains. We
do this task in subsection 2.B.

\textbf{2.A. Homological dimensions of modules over certain
non-Noetherian rings.} By $\aleph_{-1}$ we denote the cardinality of
finite sets. By $\aleph_0$ we mean the cardinality of the set of all
natural numbers. Consider the set $\Omega:=\{\alpha:\alpha  \textmd{
is a countable ordinal }\footnote{Recall that a set $X$ is an
ordinal if X is totally ordered with respect to inclusion and every
element of $X$ is also a subset of $X$. Also, one can see that
$\Omega$ is itself an ordinal number larger than all countable ones,
so it is an uncountable set.} \textmd{number} \}$.   By definition
$\aleph_1$, is the cardinality of $\Omega$. Inductively, $\aleph_n$
can be defined for all $n\in\mathbb{N}$. For more details on this
notion we recommend the reader to \cite[Appendix]{O1}. A ring is
called $\aleph_n$-Noetherian if each of its ideals  can be generated
by a set of cardinality bounded by $\aleph_n$. So, Noetherian rings
are exactly $\aleph_{-1}$-Noetherian rings. In the next result we
give some basic properties of $\aleph_n$-Noetherian rings.

\begin{lemma}\label{co} The following assertions hold.
\begin{enumerate}
\item[i)] Let $A$ be a  $\aleph_{n}$-Noetherian ring and let $\textit{S}$
be a multiplicative closed subset of $A$. Then   $\textit{S}^{-1}A$
is $\aleph_{n}$-Noetherian.
\item[ii)] Let $A$ be a  $\aleph_{n}$-Noetherian ring. Then any homomorphic
image of $A$  is $\aleph_{n}$-Noetherian.
\item[iii)] Let $\{A_n:n\in \mathbb{N}\}$ be a chain of commutative
Noetherian rings and let $A:=\bigcup A_n$. Then $A$ is
$\aleph_{0}$-Noetherian. In particular,  if $A$  is Noetherian, then
the ring
$A[X_1,X_2,\cdots]:=\bigcup_{n=1}^{\infty}A[X_1,\cdots,X_n]$ is
$\aleph_{0}$-Noetherian.
\end{enumerate}
\end{lemma}
\begin{proof}
The proof is straightforward and we leave it to the reader.
\end{proof}

In what follows we will use the following remarkable results of
Osofsky several times. It is worth to recall that if $n=-1$, Lemma
\ref{j} iii) below, is just a rephrasing of the fact that over a
Noetherian ring  projective and flat dimensions  of a cycle module
coincide.

\begin{lemma}\label{j} The following assertions hold.
\begin{enumerate}
\item[i)] (\cite[Page 14]{O2})
Let $V$ be a valuation domain and let $\frak r$ be an ideal of $V$.
Then $\pd(\frak r)=n +1$ if and only if  $\frak r$ is generated by
$\aleph_n$ but no fewer elements.
\item[ii)] (\cite[Proposition 2.62]{O1}) Let $n$ be any nonnegative integer or $\infty$.
Then there exists a  valuation domain $V$ with global dimension n.
\item[iii)] (See the proof of \cite[Corollary 2.47]{O1}) Let  $\fa$ be an ideal of a $\aleph_n$-Noetherian ring
$A$. Then $\pd_{A}(A/\fa)\leq\fd_{A}(A/\fa)+n+1$.
\end{enumerate}
\end{lemma}

Recall that a ring is coherent if each of its finitely generated
ideals are finitely presented. A typical example is a valuation
domain. In the sequel  we will need the following result.

\begin{lemma}\label{g} The following assertions hold.
\begin{enumerate}
\item[i)] (\cite[Theorem 1.3.9]{Gl}) Let $A$ be a ring. Then

\[\begin{array}{ll}
\wdim A&:=\sup\{\fd(M):M \emph{ is an $A$-module }\}\\&=\sup\{\fd(A/
\fa):\fa \emph{ is a finitely generated ideal of }A\}.
\\
\end{array}\]

\item[ii)] (\cite[Corollary 2.5.10]{Gl}) Let $A$ be a coherent ring and let $M$ be a
finitely presented $A$-module. Then $\pd_A(M)\leq n$ if and only if
$\Tor_{n+1}^A(M,A/\fm)=0$ for all maximal ideals $\fm$ of $A$.

\item[iii)] (Auslander's global dimension Theorem; \cite[Theorem 2.17]{O1})
Let $A$ be a ring. Then $$\gd A=\sup\{\pd_{A}(A/
\fa):\fa\trianglelefteq A\}.$$
\end{enumerate}
\end{lemma}

\textbf{2.B. The perfect and absolute integral closure of Noetherian
domains.} Recall that the absolute integral closure $R^+$ is defined
as the integral closure of a domain $R$ inside an algebraic closure
of the field of fractions of $R$. Throughout this paper $p$ is a
prime number. In this subsection we summarize the basic  results
concerning $R^+$ and the perfect closure of $R$ (when $\Char R=p$):
$$R_{\infty}:=\{x\in
R^+:x^{p^n}\in R \textmd{ for some
}n\in\mathbb{N}\cup\{0\}\}.$$Assume that $\Char R=p$ and let $A$ be
either $R^+$ or $R_{\infty}$ and let $x\in A$. By $(x^{\infty})A$ we
mean that $(x^{1/p^n}:n\in \mathbb{N}\cup\{0\})A$. Now, assume that
$\Char R=0$ and  let $x\in R^+$. By $(x^{\infty})R^+$ we mean
$(x^{1/n}:n\in \mathbb{N})R^+$. Note that a local domain $(R,\fm,
k)$ has mixed characteristic $p$, if $\Char R=0$ and $\Char k=p$. We
say that $R$ has residue prime characteristic if $\Char k=p$. We now
list some properties of $R^+$ and $R_{\infty}$.

\begin{lemma}\label{wel}
Let $(R,\fm)$ be a  Noetherian local domain.
\begin{enumerate}
\item[i)]
There is a $\mathbb{Q}$-valued valuation map on $R^+$  which is
nonnegative on $R^+$ and positive on $\fm R^+$.
\item[ii)] Let $\epsilon$ be a real number and let
$\fa_{\epsilon}:=\{x\in R^{+}| v(x)> \epsilon\}$. Then
$\fa_{\epsilon}$ is an ideal of $R^+$.
\item[iii)]  Assume that $R$ has prime characteristic
$p$. Let $A$ be either $R^+$ or $R_{\infty}$ and let
$x_1,\cdots,x_{\ell}$ be a finite sequence of elements of $A$. Then
$\sum_{i=1}^{\ell}(x_i^{\infty})A$ is a radical ideal of $A$.
\item[iv)] Assume that $R$ has mixed characteristic  $p$.
Let $p,x_2,\cdots,x_{\ell}$ be a finite sequence of elements of
$R^+$. Then $(p^{\infty})R^++\sum_{i=2}^{\ell}(x_i^{\infty})R^+$ is
a radical ideal of $R^+$.
\item[v)] If $(R,\fm)$ is complete, then $R^{+}$
is a directed union of module-finite extensions of $R$ which are
complete, local and normal.
\item[vi)] Let  $\textit{S}$ be  any
multiplicative closed subset of $R$. Then $\textit{S}^{-1}R^+\cong
(\textit{S}^{-1}R)^+$.
\item[vii)] Assume that $R$ has residue prime characteristic.
Let $A$ be either $R^+$ or $R_{\infty}$ and let
$x_1,\cdots,x_{\ell}$ be a finite sequence of elements of $A$. Then
$\fd_A(A/\sum_{i=1}^{\ell}(x_i^{\infty})A)\leq \ell$.
\end{enumerate}
\end{lemma}

\begin{proof} i) This is in \cite[Page 28]{H2}. Note that in that
argument $R$ does not need to be complete.

ii) This is easy to check and we leave it to the reader.

For iii) and iv), see parts (1) and (2) of \cite[Propostion
2.11]{H1}.  Note that iv)  may be checked modulo $(p^{\infty})R^+$,
to translates the prime characteristic case.

v) By the proof of \cite[Lemma 4.8.1]{HS}, $R^{+}$ is a directed
union of module-finite extensions of $R$ which are complete and
local. Let $\overline{R'}$ be the integral closure of $R'$ in its
field of fractions. Recall from \cite[Theorem 4.3.4]{HS} that the
integral closure of a complete local domain in its field of
fractions is Noetherian and local. Thus $\overline{R'}$ is a
Noetherian complete local  normal domain. To conclude, it  remains
to recall that $\overline{R'}\subseteq R^+$.

vi) See the second paragraph of \cite[Section 2]{AH}.

vii) See \cite[Proposition 2.15]{H1}.
\end{proof}

We close this section by the following  corollary of Lemma
\ref{wel}.
\begin{corollary} \label{local}
Let $(R,\fm)$ be a  Noetherian local Henselian domain of dimension $d$.
\begin{enumerate}
\item[i)] Assume that $R$ has prime characteristic
$p$. Let $A$ be either $R^+$ or $R_{\infty}$ and let
$x_1,\cdots,x_d$ be a  system of parameters for $R$. Then
$\sum_{i=1}^{d}(x_i^{\infty})A$ is a maximal ideal of $A$.
\item[ii)] Assume that $R$ has mixed characteristic  $p$.
Let $p,x_2,\cdots,x_d$ be a  system of parameters for $R$. Then
$(p^{\infty})R^++\sum_{i=2}^{d}(x_i^{\infty})R^+$ is a maximal ideal
of $R^+$.
\end{enumerate}
\end{corollary}
\begin{proof}
It follows by the fact that a radical ideal $\fa$ of a ring $A$ is
maximal if $\Ht(\fa)=\dim A<\infty$.
\end{proof}

\section{A Local Global Principle Theorem}

By Auslander's global dimension Theorem, $\gd A=\sup\{\pd_{A}(A/
\fa):\fa\trianglelefteq A\}$. So, in order to study  the global
dimension of $A$, it is enough for us to commute $\pd_{A}(A/ \fa)$
for all ideals $\fa$ of $A$. It would be interesting to know whether
the same equality remains true for some special types of ideals. For
instance, we know that if $A$ is Noetherian, then $$\gd
A=\sup\{\pd_{A}(A/ \fm):\fm \in\Max A\}=\sup\{\pd_{A}(A/ \fp):\fp
\in\Spec A\}.$$ Thus, the maximal ideals, prime ideals and finitely
generated ideals might be appropriate candidate for our proposed
ideals.

\begin{definition}Let $\Sigma$ be
a subset of the set of all ideals of $A$.  We say that $A$ has
finite global dimension on $\Sigma$, if $\sup\{\pd_{A}(A/
\fa):\fa\in\Sigma\}< \infty.$
\end{definition}

\begin{lemma}\label{cor}
Let $A$ be a  $\aleph_n$-Noetherian ring. Then, $A$ has finite
global dimension on finitely generated ideals if and only if $A$ has
finite global dimension.
\end{lemma}

\begin{proof} By Auslander's global dimension Theorem, $\gd
A=\sup\{\pd_{A}(A/ \fa):\fa\trianglelefteq A\}$. By applying Lemma
\ref{j} iii),  we get that $$\pd_{A}(A/ \fa)\leq\fd_{A}(A/
\fa)+n+1.$$ Thus $\gd A\leq\wdim A+n+1$. In view of Lemma \ref{g}
i), we see that
$$\wdim A=\sup\{\fd(A/ \fa):\fa \emph{ is a finitely generated ideal
of }A\}.$$ We incorporate these observations in to  $\fd_{A}(A/
\fa)\leq\pd_{A}(A/ \fa)$ for all ideals $\fa$ of $A$. It turns out
that$$\wdim A\leq\sup\{\pd(A/ \fa):\fa \emph{ is a finitely
generated ideal of }A\}<\infty.$$This completes the proof.
\end{proof}

\begin{lemma} \label{le}Let $A$ be a coherent ring of finite
global dimension on maximal ideals. Then, $A$ has finite global
dimension on finitely generated ideals.
\end{lemma}
\begin{proof} Let $\fa$ be a finitely generated ideal of $A$. Then
$A/\fa$ is finitely presented. By Lemma \ref{g} ii),
$\pd_A(A/\fa)\leq n$ if and only if $\Tor_{n+1}^A(A/\fa,A/\fm)=0$
for all maximal ideals $\fm$ of $A$. This completes the proof.
\end{proof}
The following is our main result in this section.

\begin{theorem} \label{lg}(Local Global Principle) Let $A$
be a coherent ring which is $\aleph_n$-Noetherian for some integer
$n\geq -1$. Then the following are equivalent:
\begin{enumerate}
\item[i)] $A$ has finite global dimension,
\item[ii)] $A$ has finite  global dimension on radical ideals,
\item[iii)]$A$ has finite
global dimension on prime spectrum,
\item[iv)]  $A$ has finite  global dimension
on maximal ideals, and
\item[v)] $A$ has finite global dimension on finitely generated
ideals.
\end{enumerate}
\end{theorem}

\begin{proof} The implications $i) \Rightarrow ii)\Rightarrow iii)$ and
$iii)\Rightarrow iv)$ are trivial.

$iv) \Rightarrow v)$ This follows by  Lemma \ref{le}.

$v)\Rightarrow i)$ This is in Lemma \ref{cor}.
\end{proof}

\begin{remark} Consider the family  of coherent quasi-local $\aleph_n$-Noetherian rings
of finite global dimensions. This family contains strictly the class
of Noetherian regular local rings. Indeed, let $n\geq0$ be an
integer.  In view of Lemma \ref{j} ii), there exists a valuation
domain $A$ of global dimension $n+2$. Clearly, $A$ is coherent. By
applying Lemma \ref{j} i), we see that $A$ is $\aleph_n$-Noetherian
but not $\aleph_{n-1}$-Noetherian. In particular, $A$ is not
Noetherian.
\end{remark}

It is noteworthy to remark that the assumptions of the previous
results are really needed.

\begin{example}\label{elg}
\begin{enumerate}
\item[i)] There exists a ring $A$ of finite global dimension on finitely
generated ideals but not of finite global dimension. To see this,
let $A$  be a valuation domain of infinite global dimension. Note
that such  a ring exists, see Lemma \ref{j} ii). Since, any finitely
generated ideal of $A$ is principal,  it turns out that $A$ is
coherent and it has finite global dimension on finitely generated
ideals. Clearly, by Lemma 3.2 i), $A$ is not $\aleph_n$-Noetherian
for all integer $n\geq-1$.
\item[ii)] There exists a ring $A$ of finite global dimension on radical
ideals but not of finite global dimension on finitely generated
ideals. Let $A_0$ be the ring of polynomials with nonnegative
rational exponents in an indeterminant $x$ over a field. Let $T$ be
the localization of $A_0$ at $(x^{\alpha}:\alpha>0)$ and set
$A:=T/(x^{\alpha}u : \emph{ u is unit, } \alpha>1)$. Then by
\cite[Page 53]{O1}, $A$ has finite global dimension on maximal
ideals and $\pd(x^{1/2}A)=\infty$.  Note that $\dim A=0$, and so
$\Spec A=\Max A$. Thus, $A$ has finite global dimension on radical
ideals but not of finite global dimension on finitely generated
ideals. By  Lemma 2.1,  $A$ is $\aleph_0$-Noetherian. Note that
Lemma \ref{le} asserts that $A$ is  not coherent.
\end{enumerate}
\end{example}

Let $\fp$ be a prime ideal of a ring $A$. In \cite[Theorem 3]{N},
Northcott proved that $\rank_{A/ \fp} \fp/\fp^2 \leq \wdim A$. Here,
we give an application of this Theorem.

\begin{corollary}Let  $(V,\fm)$ be a valuation domain. Then the following are
equivalent:
\begin{enumerate}
\item[i)]  $\bigcap\fm^n=0$,
\item[ii)] $V$ is Noetherian,
\item[iii)] $V$ is a principal ideal domain, and
\item[iv)] $V$ is an unique factorization domain.
\end{enumerate}
\end{corollary}

\begin{proof}
Without loss of generality, we can assume that $V$ is not a field.

 $i)\Rightarrow ii)$
First, note that any finitely generated ideal of $V$ is principal.
Then by Lemma \ref{g} i),
$$\wdim V=\sup\{\fd(V/ \frak r):\frak r \emph{ is a finitely generated ideal
of }V\}=1.$$  By \cite[Theorem 3]{N}, $\rank_{V/ \fm}(\fm/
\fm^2)\leq\wdim V=1$. Thus, $\fm=aV+\fm^2$ for some $a\in V$. As the
ideals of a valuation domain are linearly ordered by means of
inclusion, one has either $aV\subseteq \fm^2$ or $\fm^2\subseteq
aV$. Due to the Hausdorff assumption on $\fm$ we can assume that
$\fm\neq\fm^2$, and so $\fm=aV$. In view of \cite[Exercise 3.3]{M},
$V$ is a discrete valuation domain.

 $ii)\Rightarrow iii)$ and $iii)\Rightarrow iv)$   are
well-known.

 $vi)\Rightarrow i)$ By \cite[Remark 3.13]{AT1},
  $\Ht\fm=1$. Thus $\fm = xV$ for some $x$, because $V$ is an
unique factorization domain. Therefore $V$ is Noetherian, since each
of its prime ideals are finitely generated. The argument can now by
completed by applying Krull's Intersection Theorem.
\end{proof}

\section{Projective dimension of certain modules over $R^{+}$}

Throughout this section $R$ is a domain. The aim of this section is
to establish the projective dimensions of residue ring of certain
ideals of $R^{+}$.  First we recall the following result of Auslander.

\begin{lemma}\label{Aus}(\cite[Lemma 2.18]{O1}) Let $A$ be a ring and
let $\Gamma$ be a well-ordered set. Suppose that $\{N_\gamma :
\gamma\in \Gamma\}$ is a collection of submodules of an $A$-module
$M$ such that $\gamma'\leq \gamma$ implies $N_{\gamma'}\subseteq
N_{\gamma}$ and $M = \bigcup_{\gamma\in \Gamma} N_\gamma$. Suppose
that $\pd_A (N_\gamma/\bigcup_{\gamma'<\gamma} N_{\gamma'})\leq n$
for all $\gamma\in \Gamma$. Then $pd_A(M)\leq n$.
\end{lemma}

\begin{lemma} \label{key}Let $R$ be a domain and
let $x$ be a nonzero and nonunit element of $R^+$.
\begin{enumerate}
\item[i)] If $R$ is of characteristic zero, then $(x^{\infty})R^+$
has a bounded free resolution of countably generated free
$R^+$-modules of length one. In particular,
$\pd_{R^+}((x^{\infty})R^+)\leq1$. The equality holds, if $R$ is
Noetherian Henselian and local.
\item[ii)] If $R$ is of prime characteristic $p$, then $(x^{\infty})R^+$
has a bounded free resolution of countably generated free
$R^+$-modules of length one. The equality holds, if $R$ is
Noetherian Henselian and local.
\end{enumerate}
\end{lemma}

\begin{proof} i)
Clearly, $\frac{1}{n}-\frac{1}{n+1}>0$, and so
$x^{\frac{1}{n}-\frac{1}{n+1}}\in R^+$. Note that
$x^{1/n}=x^{1/n+1}x^{\frac{1}{n}-\frac{1}{n+1}}$. In particular,
$x^{1/n}R^+\subseteq x^{1/n+1}R^+$. For each $n$, the exact sequence
$$0\longrightarrow x^{1/n}R^+\longrightarrow x^{1/n+1}R^+
\longrightarrow x^{1/n+1}R^+ / x^{1/n}R^+ \longrightarrow 0$$ is a
projective resolution of $x^{1/n+1}R^+ / x^{1/n}R^+$. So $\pd_{R^+}(
x^{1/n+1}R^+ / x^{1/n}R^+)\leq 1$. By  Lemma \ref{Aus},
$\pd_{R^+}((x^{\infty})R^+)\leq1$. Now, we construct the concrete
free resolution of $(x^{\infty})R^+$. Let $F_0$ be a free
$R^+$-module with base $\{e_n:n\in \mathbb{N}\}$. The assignment
$e_n\mapsto x^{1/n}$ provides a natural epimorphism
$\varphi:F_0\longrightarrow(x^{\infty})R^+$. We will show that
$\ker\varphi$ is free over $R^+$. For each integer $n$ set
$\eta_n:=e_n-x^{1/n(n+1)}e_{n+1}$. Let $F_1$ be a submodule of
$\ker\varphi$ generated by $\{\eta_n:n\in \mathbb{N}\}$. It is easy
to see that $F_1\subseteq\ker\varphi$. In fact, the equality holds.
Let $\Sigma_{i=1}^n \alpha_i e_i$ be in $\ker\varphi$, where
$\alpha_i\in R^+$ for all $1\leq i\leq n$. It is straightforward to
check that $$\alpha_n x^{1/n}=-\Sigma_{i=1}^{n-1}\alpha_i
x^{1/i}=\beta_{n-1}x^{1/n-1}=\beta_{n-1}x^{1/n}x^{1/n^2-n},$$ where
$\beta_{n-1}:=-(\alpha_1 x^{\frac{n-2}{n-1}}+\alpha_2
x^{\frac{n-3}{2(n-1)}}+\alpha_3
x^{\frac{n-4}{3(n-1)}}+\cdots+\alpha_{n-1})$. Thus
$\alpha_n=\beta_{n-1}x^{1/n^2-n}$, because $R^+$ is an integral
domain. Hence
$$\Sigma_{i=1}^n \alpha_i
e_i+\beta_{n-1}\eta_{n-1}=\Sigma_{i=1}^{n-2} \alpha_i
e_i+(\alpha_{n-1}+\beta_{n-1})e_{n-1}\in \ker\varphi.$$ By using
induction on $n$, one deduces that $\Sigma_{i=1}^{n-2} \alpha_i
e_i+(\alpha_{n-1}+\beta_{n-1})e_{n-1}\in F_1$. This yields that
$\Sigma_{i=1}^n \alpha_i e_i\in F_1$,  because
$\beta_{n-1}\eta_{n-1}\in \ker\varphi$. Thus $\ker\varphi$ is
generated by the set $\{\eta_n|n\in\mathbb{N}\}$. In order to
establish $\pd_{R^+}((x^{\infty})R^+)\leq1$, it is therefore enough
for us to prove that $F_1$ is a free $R^+$-module with base
$\{\eta_n:n\in \mathbb{N}\}$. Assume that $\Sigma_{i=1}^n \alpha_i
\eta_i=0$, where $\alpha_i \in R^+$. View this equality in $F_0$.
The coefficient of $e_{n+1}$ in the left hand side of the equality
is $-x^{1/n^2+n}\alpha_n$. So $\alpha_n=0$. Continuing inductively,
we get that $\alpha_{n-1}=\cdots =\alpha_{1}=0$. Hence $F_1$ is a
free $R^+$-module. This yields the desired claim.\footnote{Note that
we proved $\pd_{R^+}((x^{\infty})R^+)\leq1$ by two different
methods. The first one is an easy application of a result Auslander,
see Lemma \ref{Aus}. The second actually proved a more stronger
result: $(x^{\infty})R^+$ has a free resolution of countably
generated free $R^+$-modules of length bounded by one.}

Now, we assume that $R$ is Noetherian local and Henselian. It turns
out that $R^+$ is quasi-local. If $\pd_{R^+}((x^{\infty})R^+)=1$ was
not be the case, then $(x^{\infty})R^+$ should be projective and
consequently free, see \cite[Theorem 2.5]{M}. Over a domain $A$ an
ideal $\fa$ is free if and only if it is principal. Indeed, suppose
on the contrary that $\fa$ is free with the base set
$\Lambda:=\{\lambda_\gamma:\gamma\in \Gamma\}$ and $|\Gamma|\geq2$.
Let $\lambda_1,\lambda_2\in\Lambda$.  The equation
$r\lambda_1+s\lambda_2=0$ has a nonzero solution as $r:=\lambda_2$
and $s:=-\lambda_1$, a contradiction. So, $(x^{1/n}:n\in
\mathbb{N})R^+=cR^+$ for some $c\in R^+$. We adopt the notation of
Lemma \ref{wel} i).  We see that $v(r)\geq v(c)>0$ for all
$r\in(x^{\infty})R^+$. Now, let $n\in\mathbb{N}$ be such that
$v(c)>v(x)/n$. Then $v(x^{1/n})=v(x)/n<v(c)$, a contradiction.

ii) Let $F_0$ be a free $R^+$-module with base $\{e_n:n\in
\mathbb{N}\cup\{0\}\}$. The assignment $e_n\mapsto x^{1/p^n}$
provides a natural epimorphism
$\varphi:F\longrightarrow(x^{\infty})R^+$. Set
$\eta_n':=e_n-x^{\frac{p-1}{p^{n+1}}}e_{n+1}$. In the proof of i)
replace $\eta_n$ by $\eta_n'$. By making straightforward
modification of the proof i), one can check easily that $\pd_{R^+}
(\fa)\leq1$.
\end{proof}

Now, we establish another preliminary lemma.

\begin{lemma} \label{gen} Let $(R,\fm)$ be a quasi-local domain of residue
prime characteristic $p$. Let $x_1,\cdots,x_{\ell}$ be a finite
sequence of nonzero and nonunit elements of $R^+$. In mixed
characteristic case assume in addition that $x_1=p$. Then
$R^+/\sum_{i=1}^{\ell}(x_i^{\infty})R^+$ has a bounded free
resolution of countably generated free $R^+$-modules of length
$2\ell$.
\end{lemma}

\begin{proof}  In view
of  Lemma \ref{key}, $R^+/ (x_i^{\infty})R^+$ has a  bounded free
resolution $\textbf{Q}^i$ of countably generated free $R^+$-modules
of length $2$. By using induction on $\ell$, we will show that
$\bigotimes_{i=1}^{\ell} \textbf{Q}^i$  is a  bounded free
resolution of $R^+/\sum_{i=1}^{\ell}(x_i^{\infty})R^+$ consisting of
countably generated free $R^+$-modules of length  at most $2\ell$.
By the induction hypothesis, $\textbf{P}:=\bigotimes_{i=2}^{\ell}
\textbf{Q}^i$ is a  bounded free resolution of
$R^+/\sum_{i=2}^{\ell}(x_i^{\infty})R^+$ consisting of countably
generated free $R^+$-modules of length  at most $2\ell-2$. In light
of \cite[Theorem 11.21]{R}, we see that
$$H_n(\textbf{Q}^1\otimes_{R^+}\textbf{P})\cong \Tor^{R^+}_n(R^+/
\fa_1,R^+/ \sum_{i=2}^{\ell}(x_i^{\infty})R^+).$$ The ideal
$(x_1^{\infty})R^+$ is a directed union of free  $R^+$-modules, and
so it is a flat $R^+$-module. Hence, for each $n>1$ we have
$\Tor^{R^+}_n(R^+/ (x_1^{\infty})R^+,R^+/
\sum_{i=2}^{\ell}(x_i^{\infty})R^+)=0$. Also,
$$\Tor^{R^+}_1(R^+/ (x_1^{\infty})R^+,R^+/
\sum_{i=2}^{\ell}(x_i^{\infty})R^+)
\cong (x_1^{\infty})R^+\cap\sum_{i=2}^{\ell}(x_i^{\infty})R^+/
(x_1^{\infty})R^+\sum_{i=2}^{\ell}(x_i^{\infty})R^+,$$ which is
zero by \cite[Propostion 2.11(3)]{H1}. On the other hand,
$$H_0(\textbf{Q}^1\otimes_{R^+}\textbf{P})\cong
R^+/(x_1^{\infty})R^+\otimes_{R^+}R^+/\sum_{i=2}^{\ell}(x_i^{\infty})R^+\cong
R^+/ \sum_{i=1}^{\ell}(x_i^{\infty})R^+.$$ This implies that
$\bigotimes_{i=1}^{\ell}
\textbf{Q}^i=\textbf{Q}^1\otimes_{R^+}\textbf{P}$  is a  bounded
free resolution of $R^+/\sum_{i=1}^{\ell}(x_i^{\infty})R^+$
consisting of countably generated free $R^+$-modules of length  at
most $2\ell$, which is precisely what we wish to prove.
\end{proof}

We immediately exploit Lemma \ref{gen} to prove some parts of
Theorem 1.1.

\begin{theorem} Let $(R,\fm)$ be a  Noetherian Henselian local domain of residue
prime characteristic $p$.
\begin{enumerate}
\item[i)] The $R^+$-module
$R^+/ \fm_{R^+}$ has a free resolution of countably generated free
$R^+$-modules of length bounded by $2\dim R$.
\item[ii)] If $\dim R=1$, then $\pd_{R^+}(R^+/\fm_{R^+})= 2$.
\end{enumerate}
\end{theorem}

\begin{proof} i) Let $d:=\dim R$ and let $\{x_1,\cdots,x_d\}$ be a system of parameters
for $R$. In the mixed characteristic case, we may and do assume in
addition that $x_1=p$. By Lemma \ref {gen},
$R^+/\sum_{i=1}^{d}(x_i^{\infty})R^+$ has a free resolution of
countably generated free $R^+$-modules of length bounded by $2d$. It
remains to recall from Corollary \ref{local} that
$\sum_{i=1}^{d}(x_i^{\infty})R^+=\fm_{R^+}$.

ii) This follows by part i) and Corollary \ref{local}.
\end{proof}

Here, we give some more examples of ideals of $R^+$ of finite
projective dimension.

\begin{example}
\begin{enumerate}
\item[i)] Let $(R,\fm)$ be a Noetherian complete local domain of
prime characteristic. Let $\fa$ be a finitely generated ideal of
$R^+$ with the property that $\Ht \fa\geq \mu(\fa)$, where
$\mu(\fa)$ is the minimal number of elements that needs to generate
$\fa$. We show that $\pd_{R^+}(R^+/ \fa)\leq\mu(\fa)\leq \dim R$.
Indeed, let $\underline{a}:=a_1,\cdots,a_n$ be a minimal generating
set for $\fa$. Keep in mind that $R^{+}$  is a directed union of its
subrings $R'$ which are module-finite extensions of $R$. In view of
Lemma \ref{wel} v), $R'$ is a complete local domain. Let $R'$ be one
of them, which contains  $a_i$ for all $1\leq i \leq n$. Set
$\fb:=(a_1,\cdots, a_n)R'$. Then $\fb R^+=\fa$. Also, we have
$n\leq\Ht \fa\leq\Ht \fb\leq\mu(\fa)\leq n$, because $R^+$ is an
integral extension of  a Noetherian ring $R'$. So $n:= \mu(\fb)=\Ht
\fb$. In view of the equality $(R')^{+}=R^{+}$, we can and do assume
that $a_i\in R$ for all $1\leq i \leq n$. This yields that
$\underline{a}$ is a part of a system of parameter for $R$. By
\cite[Theorem 5.15]{HH1}, $R^+$ is a balanced big Cohen-Macaulay
$R$-algebra. So $\underline{a}$ is a regular sequence on $R^+$.
Therefore, the Koszul complex of $R^+$ with respect to
$\underline{a}$ provides a projective resolution of $R^+/ \fa$  of
length $\mu(\fa)$.

\item[ii)] Let $\mathbb{F}$ be a field of prime characteristic.  Consider the
local ring $R:=\mathbb{F}[[X^2,Y^2,XY]]$. Its maximal ideal is
$\fm:=(X^2,Y^2,XY)R$. We show that $\pd_{R^+}(R^+/ \fm R^+)=2$.
Indeed, consider the complete regular local ring
$A:=\mathbb{F}[[X,Y]]$. One has $A^+=R^+$. By \cite[6.7,
Flatness]{HH1}, $A^+$ is flat over $A$. Therefore, $$\pd_{R^+}(R^+/
\fm R^+)=\pd_{A^+}(A^+/ \fm A^+)\leq\pd_A(A/ \fm A)\leq2.$$

\item[iii)] Let $k$ be the algebraic closure of $\mathbb{Z}/p\mathbb{Z}$ and let
$(R,\fm,k)$ be a 2-dimensional complete normal domain. Then \cite[Remark 4.10]{AH}
says that  every two generated ideal of $R^+$ has finite projective
dimension.
\end{enumerate}
\end{example}

Here, we give a simple application of Lemma \ref{wel} vii).

\begin{remark}
Wodzicki has constructed an example of a ring $A$ such that its
Jacobson radical $\frak r$ is nonzero and satisfies $\Tor^{A}_i (A/
\frak r, A/ \frak r)=0$ for all $i\geq1$. This class of rings are of
interest due to Telescope Conjecture, see \cite[Page 1234]{K}. Now,
assume that $(R,\fm)$ is a one-dimensional local domain of prime
characteristic. We show that $R^+$ belongs to such class of rings.
Indeed, by Lemma \ref{wel} vii), we know that $\Tor^{R^+}_i (R^+/
\fm_{R^+}, R^+/ \fm_{R^+})=0$  for all $i\geq2$ and
$$\Tor^{R^+}_1 (R^+/ \fm_{R^+}, R^+/
\fm_{R^+})=\fm_{R^+}/\fm_{R^+}^2.$$ To conclude, it remains recall
from Lemma \ref{4.2.2} that $\fm_{R^+}=\fm_{R^+}^2$. \end{remark}

\section{Global dimension of $R^{+}$ in one special case}

The aim of this section is to complete the proof of Theorem
\ref{main}. We state our results in more general setting and include
some of their converse parts.

\begin{lemma} \label{alef}Let $(R,\fm)$ be a one-dimensional  Henselian local
domain and has residue prime characteristic.  Then $R^+$ is
$\aleph_{0}$-Noetherian.
\end{lemma}

\begin{proof}  Note that $R^+$ is quasi-local and its maximal
ideal $\fm_{R^+}$ is countably generated. By adopting the proof of
Cohen's Theorem \cite[Theorem 3.4]{M}, one can prove that if every
prime ideal of a ring $A$ can be generated by a countable set, then
$A$ is $\aleph_{0}$-Noetherian. This completes the proof.
\end{proof}

\begin{lemma}\label{int} Let $R$ be a Noetherian domain  and of dimension greater than one. Then
$\bigcap_{0\neq \fp\in\Spec R^+}\fp =0$.
\end{lemma}

\begin{proof}   Let $\fa$ be the intersection
of all height one  prime ideals of $R$.  Assume that $\fa\neq0$.
Then having \cite[Theorem 31.2]{M} in mind, we find that $\fa$ has
infinitely many minimal prime ideals, which is impossible as $R$ is
Noetherian. Thus $\fa=0$, and so  $\bigcap_{0\neq \fp\in\Spec
R}\fp=0$. We now suppose that  $\bigcap_{0\neq \fp\in\Spec
R^+}\fp\neq0$ and look for a contradiction. Let $ x$ be a nonzero
element of $\bigcap_{0\neq \fp\in\Spec R^+}\fp$. Keep in mind that
$R^{+}$  is a directed union of its subrings $R'$ which are
module-finite over $R$. Let $R'$ be one of them which contains $x$.
From the equality $(R') ^+=R ^+$, we may and do assume that $x\in
R$.  Thus one can find a nonzero prime ideal $\fq\in \Spec R$ such
that $x$ does not belong to $\fq$. Since the ring extension
$R\longrightarrow R^+$ is integral, by \cite[Theorem 9.3]{M}, there
exists a prime ideal $\fp$ of $R^+$ lying over $\fq$. By the choice
of $x$, we have $x\in\fp$. In particular $x\in \fq$, a
contradiction.
\end{proof}

\begin{lemma} \label{val} Let $(R,\fm)$ be a Noetherian complete local domain.
 \begin{enumerate}
\item[i)] If $\dim R^{+}=1$, then $R^{+}$ is a valuation domain with value group $\mathbb{Q}$.
\item[ii)] If $R^{+}$ is a valuation domain, then $\dim R^{+}\leq1$.
\end{enumerate}
\end{lemma}

\begin{proof}  i) Keep in mind that $R^{+}$  is a directed
union of its subrings $R'$ which are module-finite over $R$. Let
$R'$ be one of them. In view of Lemma \ref{wel} v), $R'$ is a
complete local normal domain.  Due to \cite[Theorem 11.2]{M}, we
know that $R'$ is a discrete valuation domain. Note that $R^+=R'^+$.
Thus, $R^{+}$ is a directed union of discrete valuation domains.
Since the directed union of valuation domains is again valuation
domain, $R^{+}$ is a valuation domain. Now, we compute the value
group of $R^+$. By Cohen's Structure Theorem, we can assume that
$(R,\fm)= (V, tV)$ is a discrete valuation domain. By Lemma
\ref{wel} i), there exists a value map as $v: R^+\longrightarrow
\mathbb{Q}\cup\{\infty\}$ such that $v(t)=1$. Let $ Q(R^{+})$ be the
field of fractions of $R^{+}$ and consider $R^{+}_{v}:=\{x\in
Q(R^{+})| v(x)\geq 0\}$.  Clearly, $R^{+}\subseteq
R^{+}_{v}\subsetneqq Q(R^{+})$. In light of \cite[Exercise 10.5]{M},
that there does not exist any ring properly between $R^{+}$ and
$Q(R^{+})$. Thus $R^{+}= R^{+}_{v}$. In particular, the value group
of $R^{+}$ is $\mathbb{Q}$ (note that $v((t^{1/n})^{m})=m/n$ for all
$m/n\in\mathbb{Q}_{\geq0}$).

ii) Without loss of generality we can assume that $\dim R\neq0$. As
the ideals of valuation rings are linearly ordered, we have
$\bigcap_{0\neq \fp\in\Spec R^+}\fp \neq0$. Then by applying Lemma
\ref{int}, we get the claim. \end{proof}

The preparation of Theorem \ref{main} is finished. Now, we proceed
to the proof of it. We repeat Theorem \ref{main} to give its proof.

\begin{theorem}  Let $(R,\fm)$ be a Noetherian
local Henselian domain has residue prime characteristic.
\begin{enumerate}
\item[i)] The $R^+$-module $R^+/ \fm_{R^+}$ has a free resolution of countably generated
free $R^+$-modules of length bounded by $2\dim R$. In particular,
$\pd_{R^+}(R^+/ \fm_{R^+})\leq2\dim R$.
\item[ii)] If  $R$ is one-dimensional  and complete,  then $\gd
R^+=\pd_{R^+}(R^+/ \fm_{R^+})=2$.
\end{enumerate}
\end{theorem}

\begin{proof} i) This is in Theorem 4.6.

ii) By Lemma \ref{val} i), we know that $R^+$ is a valuation domain,
and so any finitely generated ideal of $R^+$ is principal. This
along with Lemma \ref{g} i) shows that
$$\wdim R^+=\sup\{\fd(R^+/ \fa):\fa \emph{ is a finitely generated ideal
of }R^+\}=1.$$ Also, Lemma \ref{alef} indicates that $R^+$ is
$\aleph_{0}$-Noetherian. By applying this along with  Lemma \ref{j}
iii), we find that $\gd R^+\leq2$. The proof can now be completed by
an appeal to Theorem 4.6 ii).
\end{proof}

Let $i:S\hookrightarrow R$ be a ring extension. Recall that $i$ is
called a cyclically pure extension if $\fa R \cap S=\fa$ for every
ideal $\fa$ of $S$. We close this section by the following
application of Lemma \ref{val}.

\begin{proposition} Let $(R,\fm)$ be a
one-dimensional Noetherian complete local domain which contains a
field. Then, there exists a subring $R'$ of $R$ such that $
R'\hookrightarrow R$ is not a cyclically pure extension.
\end{proposition}

\begin{proof} By Cohen's Structure Theorem, there exists a complete
regular local subring $(A,\fm_{A})$ of $R$ such that $R$ is finitely
generated over $A$ and $A$ contains a field $\mathbb{F}$. Any
finitely generated torsion-free module over the principal ideal
domain $A$ is free. Thus $R$ is  free and $R^+$  is flat as modules
over $A$. Assume that the claim holds for regular rings. Then there
is a subring $A'$ of $A$ such that $A'\hookrightarrow A$ is not
cyclically pure. It yields that $A'\hookrightarrow A\hookrightarrow
R$ is not a cyclically pure extension. Then, without loss of
generality we can assume that $A=R$ is regular. By
$w:R\longrightarrow \mathbb{Z}\bigcup\{\infty\}$ we mean the value
map of the discrete valuation ring $R$. Since $R^{+}$ is an integral
extension of $R$, $w$ can be extended to a valuation map on $R^{+}$.
We denote it by $w : R^{+} \longrightarrow \mathbb{Q}\bigcup
\{\infty\}$. Note that $w$ is positive on $\fm_{R^{+}}$. Assume that
$t$ is an uniformising element of $R$, i.e. $\fm=tR$. Let $\sum $ be
the class of all subrings $R'$ of $R$ such that $ut\notin R'$ for
all $u\in R\setminus \fm$. It is clear that $w(u)=0$ for all $u\in
R\setminus \fm$. It turns out that $\mathbb{F}\in\sum$. So
$\sum\neq\emptyset$. We can partially order $\sum$ by means of
inclusion. Thus, $\sum $ is an inductive system. Let
$\{R_{\lambda}\}_{\lambda\in\Lambda}$ be a nonempty totally ordered
subset of $\sum $ and let $R'$ be their union. In view of
definition, $R'\in\sum $. By Zorn's Lemma,  $\sum $ contains a
maximal member $R'$. Next, we show that $R'$ is our proposed ring.
To this end, we show that $t^{2},t^{3}\in R'$. For this, we show
that $R'[t^{2}],R'[t^{3}]\in\sum$. Suppose on the contrary that
$R'[t^{2}]\notin\sum$. So $ut\in R'[t^{2}]$ for some $u\in
R\setminus \fm$. Hence there exists a nonnegative integer $n$ and
$r_{0},\cdots,r_{n}\in R'$ such that
$ut=r_{0}+r_{1}t^{2}+\cdots+r_{n}t^{2n}$. This implies $r_{0}=u't$,
where either $u':=u-r_{1}t-\cdots-r_{n}t^{2n-1}$ in the case
$n\neq0$ or $u':=u$ in the case $n=0$. A contradiction we search is
that $u'\in R\setminus \fm$. Similarly, $t^{3}\in R'$. Note that
$R^{+}$ is  a flat extension of $R$, and so cyclically pure. Now,
consider the ideal $\fa:=\{\alpha \in R^{+}|w(\alpha)\geq 3 \}$. In
view of Lemma  \ref{val} i), either $\fa\subseteq t^2R^{+}$ or
$t^2R^{+}\subseteq \fa$ are necessarily true. The second possibility
not be the case, since $t^2\notin \fa$. Thus, $\fa\subseteq
t^2R^{+}$. If the extension $ R'\hookrightarrow R$, was not the
case, then we should have $ R'\hookrightarrow R^{+}$ is cyclically.
Thus $t^2R^+\cap R'=t^2R'$, and so $\fa\cap R'\subseteq t^2R'$.
Hence $t^3=rt^2$ for some $r\in R'$.  Clearly, $r=t$. This is a
contradiction by $R'\in\sum$.
\end{proof}

\section{Application: Almost zero modules}
Let $(R,\fm)$ be a Noetherian complete local domain. In view of
Lemma \ref{wel} i), there is a  valuation map  $v :
R^+\longrightarrow \mathbb{Q}\bigcup \{\infty\}$  which is
nonnegative on $R^+$ and positive on $\fm_{R^+}$. Fix  such a
valuation map and consider the following definition.

\begin{definition}\label{almost}Let $M$ be an $R^{+}$-module and let $x$ be a nonzero
element of $\fm_{R^{+}}$.\begin{enumerate}
\item[i)] (\cite[Definition \ref{main}]{RSS}) $M$ is called
almost zero with respect to $v$, if for all $m\in M$ and for all
$\epsilon
>0$ there is an element $a\in R^{+}$ with $v(a)<\epsilon$ such that
$am=0$.  The notation $\mathfrak{T}_v$ stands for the class of
almost zero $R^{+}$-modules with respect to $v$.
\item[ii)] $M$ is called almost
zero with respect to $\fm_{R^{+}}$, if $\fm_{R^{+}}M=0$. We use the
notation $\mathcal{GR}$ for the class of such modules, see
\cite[Definition 2.2]{GR}.
\item[iii)] $M$
is called almost zero with respect to $x$, if $x^{1/n}$ kills $M$
for arbitrarily large $n$. $\mathfrak{T}_x$ stands for the class of
almost zero modules with respect to $x$, see \cite[Definition
\ref{main}]{F}.
\end{enumerate}
\end{definition}

Classes of  almost zero modules are (hereditary) torsion theories:

\begin{definition} \label{tordef}Let $A$ be a ring and let
$\mathfrak{T}$ be a subclass of $A$-modules. \begin{enumerate}
\item[i)]The class
$\mathfrak{T}$ is called a Serre class, if it is closed under taking
submodules, quotients and extensions, i.e.  for any exact sequence
of $A$-modules
$$0\longrightarrow M'\longrightarrow M\longrightarrow M''\longrightarrow
0,$$the module $M\in\mathfrak{T}$ if and only if $M'\in\mathfrak{T}$
and $M''\in\mathfrak{T}$.
\item[ii)] A
Serre class which is closed under taking the directed limit of any
directed system of its objects is called a torsion theory.
\end{enumerate}\end{definition}

Let $\mathfrak{T}$ be a torsion theory of $A$-modules. For an
$A$-module $M$, let $\Sigma$ be the family of all submodules of $M$,
that belongs to $\mathfrak{T}$. We can partially order $\Sigma$ by
means of inclusion. This inductive system has an unique maximal
element, call it $t(M)$. The assignment $M$ to $t(M)$ provides a
left exact functor. We denote it by $\mathfrak{T}$. We shall denote
by $\textbf{R}^{i}\mathfrak{T}$, the $i$-th right derived functor of
$\mathfrak{T}$. To make things easier, we first recall some notions.
Consider the Gabriel filtration $\mathfrak{F}:=\{\fa\trianglelefteq
A:A/\fa\in \mathfrak{T}\}$. We define a partial order on
$\mathfrak{F}$ by letting $\fa \leq \fb$, if $\fa\supseteq\fb$ for
$\fa,\fb\in\mathfrak{F}$. Assume that $E$ is an injective
$A$-module. It is clear that $\textbf{R}^{i}\mathfrak{T}(E)$ and
$\underset{\fa\in\mathfrak{F}} {\varinjlim}\Ext^{i}_{A}(A/\fa, E)$
are zero for all $i>0$. Thus, in view of the following natural
isomorphism
$$\textbf{R}^{0}\mathfrak{T}(M)=\{m\in M:\fa m=0,\exists
\fa\in\mathfrak{F}\}\cong\underset{\fa\in\mathfrak{F}}
{\varinjlim}\Hom_{A}(A/\fa, M),$$ one can find that
$\textbf{R}^i\mathfrak{T}(-)\cong\underset{\fa\in\mathfrak{F}}
{\varinjlim}\Ext^{i}_{A}(A/\fa, -)$. Let $\fa$ be an ideal of $A$.
By $\mathfrak{T}_{\fa}$, we mean that $\{M:M_{\fp}=0, \forall
\fp\in\Spec A \setminus\V(\fa)\}$. From the above computations, one
can see that for Grothendieck's local cohomology modules with
respect to $\fa$, there are isomorphisms
$$H^{i}_{\fa}(M)=\textbf{R}^i\mathfrak{T}_{\fa}(M)\cong
\underset{n}{\varinjlim}\Ext^{i}_{A} (A/\fa^{n}, M).$$ An important
numerical invariant  associated to $\mathfrak{T}$ is
$\textit{cd}(\mathfrak{T})$, the cohomological dimension of
$\mathfrak{T}$, which is the largest integer $i$ such that
$\textbf{R}^{i} \mathfrak{T}\neq0$.

\begin{proposition} \label{x}Let  $(R,\fm)$ be a Henselian local domain.
\begin{enumerate}
\item[i)]
If $(R,\fm)$  has residue prime characteristic, then
$\cd(\mathcal{GR})\leq2\dim R$.
\item[ii)]  If $x$ is a nonzero element of $\fm_{R^+}$, then $\cd
(\mathfrak{T}_x)=2$.
\item[iii)]
Let $\mathfrak{F}_v:=\{\fa\trianglelefteq R^+:R^+/\fa\in
\mathfrak{T}_v\}$. Then
$\textbf{R}^i\mathfrak{T}_v(-)\cong\underset{\fa\in\mathfrak{F}_v}
{\varinjlim}H^{i}_{\fa}(-)$. In particular,
$$\cd(\mathfrak{T}_v
)\leq\sup\{\sup\{i\in\mathbb{N}\cup\{0\}: H^{i}_{\fa}(-)\neq0 \}
|\fa\in\mathfrak{F}_v\}.$$
\end{enumerate}
\end{proposition}

\begin{proof} i) Consider the
$\mathfrak{F}_{\mathcal{GR}}:=\{\fa\trianglelefteq R^+:R^+/
\fa\in\mathcal{GR}\}$. Clearly,
$\mathfrak{F}_{\mathcal{GR}}=\{\fm_{R^+},R^+\}$. It turns out that
$$\textbf{R}^i(\mathcal{GR})(-)\cong
\underset{\fa\in\mathfrak{F}_{\mathcal{GR}}}{\varinjlim}\Ext^{i}_{R^+}(R^+/\fa,
-)\cong \Ext^{i}_{R^+}(R^+/ \fm_{R^+}, -).$$So
$\cd(\mathcal{GR})=\pd_{R^+}({R^+}/ \fm_{R^+})$.  The proof can now
be completed by an appeal to Theorem \ref{main}.

ii) Let $M$ be an almost zero $R^+$-module with respect to $x$. In
view of Definition \ref{almost} iii),  there are sufficiently large
integers $\ell$ such that $x^{1/\ell} M=0$. Assume that $\ell\neq1$.
Then $x^{\frac{1}{\ell-1}}
M=x^{\frac{1}{\ell(\ell-1)}}x^{\frac{1}{\ell}}M=0$. Continuing
inductively, $x^{1/n} M=0$ for all $n\in\mathbb{N}$. Consider the
Gabriel filtration $\mathfrak{F}_x:=\{\fa\trianglelefteq R^+:R^+/
\fa\in\mathfrak{T}_x\}$. Observe that
$\mathfrak{F}_x=\{\fa\trianglelefteq R^+:
(x^{\infty})\subseteq\fa\}$. Recall that the partial order on
$\mathfrak{F}_x$ is defined  by letting $\fa \leq \fb$, if
$\fa\supseteq\fb$ for $\fa,\fb\in\mathfrak{F}_x$. Thus
$\{(x^{\infty})\}$ is cofinal with $\mathfrak{F}_x$. Therefore,
$$\textbf{R}^i\mathfrak{T}_x(-)\cong\underset{\fa\in\mathfrak{F}_x}
{\varinjlim}\Ext^{i}_{R^+}(R^+/\fa, -)\cong
\Ext^{i}_{R^+}(R^+/(x^{\infty}), -),$$ and so $\cd
(\mathfrak{T}_x)=\pd_{R^+}({R^+}/ (x^{\infty}))$. This along with
Lemma \ref{key} completes the proof of ii).

iii) Let $\fa\in\mathfrak{F}_v$ and $n\in\mathbb{N}$. Assume that
$\epsilon >0$ is a rational number.  There is an element $a$ in
$\fa$ with $v(a)<\epsilon/n$. Consequently, $v(a^n)<\epsilon$. Thus,
$\fa ^n$ has elements of small order, i.e. $\fa
^n\in\mathfrak{F}_v$. Therefore,
$$\textbf{R}^i\mathfrak{T}_v(M)\cong\underset{\fa\in\mathfrak{F}_v}{\varinjlim}
\Ext^{i}_{R^+}(R^+/\fa,
M)\cong\underset{\fa\in\mathfrak{F}_v}{\varinjlim}
(\underset{n\in\mathbb{N}}{\varinjlim}\Ext^{i}_{R^+}(R^+/\fa^n,
M))\cong\underset{\fa\in\mathfrak{F}_v}{\varinjlim}H^{i}_{\fa}(M),$$as
claimed.
\end{proof}

Each valuation map  may gives a different class of almost zero
modules over $R^{+}$.   However, in the case $\dim R=1$, one has the
following.

\begin{proposition}\label{izu}
Let $(R,\fm)$ be a one-dimensional Noetherian complete local domain.
If $x\in\fm_{R^+}$ is nonzero, then
$\mathfrak{T}_v=\mathcal{GR}=\mathfrak{T}_x$. In particular,
$\cd(\mathfrak{T}_v)=2$.
\end{proposition}
\begin{proof} Clearly,
$\mathcal{GR}\subseteq\mathfrak{T}_x\subseteq\mathfrak{T}_v$. Let
$M\in\mathfrak{T}_v$ be nonzero and let $0\neq m\in M$. We show that
$\fm_{R^+}=(0:_{R^{+}}m)$. This is trivial that
$(0:_{R^{+}}m)\subseteq \fm_{R^+}$. Assume that $x\in\fm_{R^+}$and
observe that $R^{+}/(0:_{R^{+}}m)\cong R^{+}m\in\mathfrak{T}_v$. By
Lemma \ref{val} i), either $(0:_{R^{+}}m)\subseteq xR^+$ or
$xR^+\subseteq(0:_{R^{+}}m)$. If $(0:_{R^{+}}m)\subseteq xR^+$ was
the case, then by the natural epimorphism
$R^+/(0:_{R^{+}}m)\longrightarrow R^+/xR^+$, we should have
$R^+/xR^+\in\mathfrak{T}_v$. This implies that $xR^+$ contains
elements of sufficiently small order, contradicting the fact that
$v(rx)=v(r)+v(x)\geq v(x)$ for all $r\in R^+$. So,
$xR^+\subseteq(0:_{R^{+}}m)$. Thus,
$\fm_{R^+}\subseteq(0:_{R^{+}}m)$. Hence $\fm_{R^+}M=0$, and so
$M\in\mathcal{GR}$.
\end{proof}

Let $A$ be a ring and $\mathfrak{T}$ a torsion theory of
$A$-modules. Recall that any $A$-module $M$ has a largest submodule
$t(M)$ which belongs to $\mathfrak{T}$. The $A$-module $M$ is called
torsion-free with respect to $\mathfrak{T}$, if $t(M)=0$. The
collection of all torsion-free modules with respect to
$\mathfrak{T}$ is denoted by $\mathcal{F}$. By
$\mathfrak{T}-\depth_{A}(M)\geq n$    we mean that there is an
injective resolution of $M$ such that whose first $n$-terms are
torsion-free with respect to $\mathfrak{T}$, see \cite[Definition
\ref{main}]{Ca}.  In light of \cite[Proposition 1.5]{Ca}, we see
that
$$\mathfrak{T}-\depth_{A}
(M)=\inf\{i\in\mathbb{N}\cup\{0\}:\textbf{R}^i\mathfrak{T}
(M)\neq0\}.$$Here $\inf$ is formed in $\mathbb{Z} \cup \{\infty\}$
with the convention that $\inf \emptyset=\infty$.

\begin{proposition}Let $(R,\fm)$ be a Noetherian Henselian local domain
and let $M$ be an $R^+$-module such that $\Ext^i_{R^+}(R^+/
\fm_{R^+},M)\neq 0$ for some $i$.
\begin{enumerate}
\item[i)] $\mathfrak{T}_v-\depth_{R^+} (M)\in\{0,1,2\}$.
\item[ii)] Let $\mathfrak{T}$ be any nonzero torsion theory of
$R^+$-modules. If either $R$ has prime characteristic or has mixed
characteristic, then $\mathfrak{T}-\depth_{R^+} (M)\leq 2\dim R$.
\end{enumerate}
\end{proposition}

\begin{proof}i) Recall  that $\mathcal{GR}\subseteq\mathfrak{T}_x\subseteq \mathfrak{T}_v$.
This yields that
$$\mathfrak{T}_v-\depth_{R^+} (M)\leq\mathfrak{T}_x-\depth_{R^+}
(M)\leq\mathcal{GR}-\depth_{R^+} (M).$$ By Proposition \ref{x} ii),
it is enough for us to prove that $\mathcal{GR}-\depth_{R^+}
(M)<\infty$. But this holds, because $\Ext^i_{R^+}(R^+/
\fm_{R^+},M)\neq 0$ for some $i$.

ii) Let $L\in\mathfrak{T}$ be a nonzero $R^+$-module and let $\ell$
be a nonzero element of $L$. So $R^{+}\ell\in\mathfrak{T}$. From the
natural epimorphism $R^{+}\ell\cong
R^{+}/(0:_{R^{+}}\ell)\longrightarrow R^{+}/\fm_{R^{+}},$ we get
that $R^{+}/\fm_{R^{+}}\in\mathfrak{T}$. It turns out that
$\mathcal{GR}$ is minimal with respect to inclusion, among all
nonzero torsion theories. At this point, ii) becomes clear from the
proof of i). \end{proof}

\begin{lemma} \label{6.6.1}
Let $(R,\fm)$ be a Noetherian local domain. If $0\neq\fp\in\Spec
R^{+}$, then $R^{+} / \fp \in \mathfrak{T}_v$.
\end{lemma}

\begin{proof}
This is proved in \cite{AT2}. But for the sake of completion, we
present  its short proof here. Let $x\in \fp$. For any positive
integer $n$ set $f_n(X):=X^n- x\in R^+[X]$. Let $\zeta_n$ be a root
of $f_n$ in $R^+$. It follows that $\zeta_n\in \fp$, since
$(\zeta_n)^n=x\in\fp$. Keep in mind that $v$ is positive on $\fp$.
The equality $v(\zeta_n)=v(x)/n$ indicates that $\fp$ has elements
of small order. Therefore, $R^{+} / \fp \in \mathfrak{T}_v$.
\end{proof}

The notation $\mathcal{F}_v$ stands for the class of torsion-free
$R^+$-modules with respect to $\mathfrak{T}_v$. Let $\epsilon$ be a
real number. Recall from Lemma \ref{wel} ii) that
$\fa_{\epsilon}:=\{x\in R^{+}| v(x)> \epsilon\}$ is an ideal of
$R^+$.

\begin{lemma}
\label{6.6.2}  Let $(R,\fm)$ be a Noetherian local domain.
\begin{enumerate}
\item[i)] Let $\epsilon$ be a positive nonrational real number. Then $R^{+}/
\fa_{\epsilon}\in\mathcal{F}_v$.
\item[ii)] Assume that  $R$ is  complete and has prime
characteristic. If $x$ is a nonzero element of $R^+$, then
$R^{+}/x{R^+}\in\mathcal{F}_v$.
\end{enumerate}
\end{lemma}

\begin{proof} i) If $R^{+}/ \fa_{\epsilon}\notin\mathcal{F}_v$, then there exists a
nonzero element $x+\fa_{\epsilon}$ 'say in the torsion part of
$R^{+} / \fa_{\epsilon}$ with respect to $\mathfrak{T}_v$. It
follows that $v(x)\leq\epsilon$. As $\epsilon$ is nonrational, one
has $v(x)<\epsilon$. By Definition 6.1 i), there exists an element
$a$ in $R^{+}$ such that $ax\in \fa_{\epsilon}$ and
$v(a)<\epsilon-v(x)$. Thus, $\epsilon < v(ax)=v(a)+v(x)<\epsilon$.
This is a contradiction.

ii) The ring $R^{+}$ is a directed union of module-finite extensions
$R'$ of $R$. So $x\in R'$ for some of them. We can assume that $R'$
is complete local and normal, see Lemma \ref{wel} v).  Since $(R')
^+=R^+$, without loss of generality we can replace $R$ by $R'$.
Thus, we may and do assume that $x\in R$. Let $y+xR^+$ be a torsion
element of $R^+/xR^+$ with respect to $\mathfrak{T}_v$. Without loss
of generality,  we can assume that $y\in R$. There are elements
$a_n\in R^+$ of arbitrarily small order such that $a_n y\in xR^+$.
\cite[Theorem 3.1]{HH2} asserts that $y\in (xR)^{\ast}$, the tight
closure of $xR$. By \cite[Corollary 10.2.7]{BH},
$(xR)^{\ast}=\overline{xR}$, the integral closure of $xR$. Recall
from \cite[Proposition 10.2.3]{BH} that $xR=\overline{xR}$, since
$R$ is normal. So $y+x R^+ =0$. Therefore, $R^+/xR^+$ is
torsion-free with respect to $\mathfrak{T}_v$.
\end{proof}

Let $A$ be a ring and let $M$ be an $A$-module. Recall that a prime
ideal $\fp$ is said to be associated to $M$ if $\fp=(0 :_A m)$ for
some $m \in M$. We denote the set of all associated prime ideals of
$M$ by $\Ass_{A}(M)$. The following result is an application of
almost zero modules.

\begin{proposition} Let $(R,\fm)$ be a Noetherian local domain which is not a field
 and let $\fp$ be a prime ideal of $R^{+}$.

 \begin{enumerate}
\item[i)] Let $\epsilon$ be a  positive nonrational real
number. Then $\fa_{\epsilon}$ is not finitely generated.

\item[ii)]  Adopt the above notation and assumptions. Then
$\Hom_{R^{+}}(R^{+} / \fp,R^{+} / \fa_{\epsilon})\neq0$ if and only
if $\fp=0$. In particular, $\Ass_{R^{+}}(R^{+} /
\fa_{\epsilon})=\emptyset$.
\item[iii)] Assume that  $R$ is  complete and has prime
characteristic and let $x\in R^{+}$ be a nonzero. Then
$\Hom_{R^{+}}(R^{+} / \fp,R^{+} / xR^{+})\neq0$ if and only if
$\fp=0$. In particular,  $\Ass_{R^{+}}(R^{+}/x{R^+})=\emptyset$.
\end{enumerate}
\end{proposition}

\begin{proof}
i) Suppose on the contrary that $\fa_{\epsilon}$ has a finite
generating set $\{x_{1},\cdots ,x_{n}\}$. Set $\alpha:= \min
\{v(x_{i})| 1\leq i\leq n\}$. Then $\epsilon <\alpha$. Let $\delta$
be a rational number, strictly  between $\epsilon$ and  $\alpha$. As
we saw in the proof of Lemma \ref{val} i), there is an element $a$
in $R^{+}$ such that $v(a)=\delta$. So $a\in \fa_{\epsilon}$. Thus,
$a=\Sigma_{i=1}^n r_i x_i$ for some $r_i\in R^+$. It turns out that
\[\begin{array}{ll}
v(a)&\geq\min\{v(r_ix_i):1\leq i \leq n\}\\&=\min\{v(r_i)+v(x_i)
:1\leq i \leq n\}\\&\geq \min\{v(x_i):1\leq i \leq n\}\\&
>\delta.\\
\end{array}\]
This  contradiction shows that $\fa_{\epsilon}$ is not finitely
generated.

ii) Recall that the pair $(\mathfrak{T}_v,\mathcal{F}_v)$ is a
maximal pair with having the property that $\Hom_{R^{+}}(T,F)=0$ for
all $T\in\mathfrak{T}_v$ and $F\in \mathcal{F}_v$.  Incorporated
this observation and Lemma \ref{6.6.2} along with Lemma \ref{6.6.2}
to prove the first claim.

Now, we show that $\Ass_{R^{+}}(R^{+} / \fa_{\epsilon})=\emptyset$.
Suppose on the contrary that $\Ass_{R^{+}}(R^{+} / \fa_{\epsilon})$
is nonempty and look for a contradiction. Let $\fp\in\Ass_{R^{+}}
(R^{+}/\fa_{\epsilon})$. Then $R^+/\fp\hookrightarrow
R^+/\fa_{\epsilon}$. Due to the first claim we know that $\fp=0$. On
the other hand $\dim_{R^+}(R^{+}/\fa_{\epsilon})<\dim_{R^+}(R^+)$.
This provides a contradiction.

iii)  By Lemma \ref{6.6.2} ii), the proof of the  claim is a
repetition of the proof of ii).
\end{proof}

 We close this section by the following
corollary.

\begin{corollary} Let $(R,\fm)$ be a $1$-dimensional
complete local domain of prime characteristic. Then
$\Ass_{R^{+}}(R^{+}/\fa)=\emptyset$ for any nonzero finitely
generated ideal $\fa$ of $R^{+}$.
\end{corollary}

\begin{proof} By Lemma \ref{val} $R^{+}$ is a valuation domain. So
its finitely generated ideals are principal. Therefore, the claim
follows by Proposition 6.8.
\end{proof}

\section{Perfect subrings of $R^+$}

In this section we deal with a subring of $R^+$ with a lot of
elements  of small order. Let $A$ be an integral domain with a
valuation map $v:A\longrightarrow \mathbb{R}\bigcup \{\infty\}$
which is nonnegative on $A$.  We say that an $A$-module $M$ is
almost zero with respect to $v$, if for all $m\in M$ and all
$\epsilon
>0$, there is an element $a\in A$ with $v(a)<\epsilon$ such that
$am=0$.  The notation $\mathfrak{T}^A_v$ stands for the class of
almost zero $A$-modules with respect to $v$.

\begin{lemma}Adapt the above notation and assumptions.
If $\mathfrak{T}^A_v\neq\emptyset$, then $A$ is a non-Noetherian
ring.
\end{lemma}

{\bf Proof.} Let $M$ be an almost zero $A$-module with respect to
$v$.  The existence of a such module  provides a sequence
$(a_n:n\in\mathbb{N})$ of elements of $A$ such that
$$\cdots<v(a_{n+1})<v(a_{n})<\cdots<v(a_{1})<1.$$ Consider the
following chain of ideals of $A$:
$$(a_1)\subsetneqq(a_1,a_2)\subsetneqq\cdots\subsetneqq(a_1,\cdots,a_n)\subsetneqq\cdots.$$
Note that, because $v$ is nonnegative on $A$, one has
$a_{n+1}\notin(a_1,\cdots,a_n)$ for all $n\in\mathbb{N}$.  $\Box$

In the following we recall the definition of $R_{\infty}$.
\begin{definition}\label{definf}
Let $(R,\fm,k)$ be a Noetherian local domain.
\begin{enumerate}
\item[i)] Assume that $\Char R=p$. By $R_{\infty}$ we denote $\{x\in R^+|x^{p^n}\in R
\textmd{ for some }n\in\mathbb{N}\cup\{0\}\}$.
\item[ii)] Assume that $R$ is  complete regular and has mixed
characteristic $p$. First, consider the case $p\notin \fm^2$. Due to
Cohen's Structure Theorem, we know that $R$ is of the form
$V[[x_2,\cdots,x_d]]$ for a discrete valuation ring $V$. By Faltings
algebra, we mean that
$$R_{\infty}:=\underset{n}{\varinjlim}V[p^{1/{p^n}}][[x_2^{1/{p^n}},\cdots,x_d^{1/{p^n}}]].$$
Now, consider the case $p\in \fm^2$. Then  Cohen's Structure Theorem
gives a discrete valuation ring $V$ and an element $u$ of
$\fm\setminus\fm^2$ such that $R$ is of the form
$V[[x_2,\cdots,x_d]]/(u)$. By $R_{\infty}$ we mean that
$$R_{\infty}:=\underset{n}{\varinjlim}V[p^{1/{p^n}}][[x_2^{1/{p^n}},\cdots,x_d^{1/{p^n}}]]/(u).$$
\item[iii)] Assume that $R$ is complete regular and has
equicharacteristic zero, i.e. $\Char R=\Char k=0$. Then by Cohen's
Structure Theorem, $R$ is of the form $k[[x_1,\cdots,x_d]]$. Take
$p$ be any prime number. By $R_{\infty}$ we mean that
$\underset{n}{\varinjlim} k[[x_1^{1/{p^n}},\cdots,x_d^{1/{p^n}}]]$.
\end{enumerate}
\end{definition}

Theorem 7.10 is  our main result in this section. To prove it, we
need a couple of lemmas.

\begin{lemma} \label{alef0}
Let $(R,\fm)$ be as Definition \ref{definf}. Then $R_{\infty}$ is
$\aleph_0$-Noetherian.
\end{lemma}

\begin{proof} First assume that $\Char R= 0$. Then by Definition 7.2,
$R_{\infty}$ is the directed union of Noetherian rings. In view of
Lemma 2.1 iii), $R_{\infty}$ is $\aleph_0$-Noetherian. Now assume
that $\Char R= p$ For each positive integer $n$, set $R_n:=\{x\in
R_\infty|x^{p^n}\in R\}$. Thus $R_\infty=\bigcup R_n$. 
It is easy to see that $R_n$ is Noetherian.
 In view of Lemma 2.1 iii), $R_{\infty}$ is
$\aleph_0$-Noetherian.
\end{proof}

\begin{lemma} (\cite[Theorem 2.3.3]{Gl})\label{dir1}
Let $\{R_\gamma:\gamma\in \Gamma\}$ be a  directed system of rings
and let $A=\underset{\gamma\in \Gamma}{\varinjlim} R_\gamma$.
Suppose that for $\gamma\leq \gamma'$, $R_{\gamma'}$ is  flat over
$R_{\gamma}$ and that $R_{\gamma}$ is a coherent ring for all
$\gamma\in \Gamma$, then $A$ is a coherent ring.
\end{lemma}

\begin{remark}\label{car}(\cite[VI, Exercise 17]{CE}) Let $\{A_n\}$ be
 a directed system of rings with directed limit $A$. For
each $n$ let $M_n$ and $N_n$ be $A_n$-modules. Assume that $\{M_n\}$
and $\{N_n\}$ are  directed  systems. Then the respective directed
limits $M$ and $N$ are $A$-modules, and for all $i$ we may identify
$\Tor_i^A(M,N)$ with the directed limit of the modules
$\Tor_i^{A_n}(M_n,N_n)$.
\end{remark}

\begin{lemma} \label{dir} The following assertions hold.
\begin{enumerate}
\item[i)] Let $\{(R_n,\fm_n):n\in \mathbb{N}\}$ be a directed chain of
Noetherian local rings such that for each $n<m$, $R_m$ is flat over
$R_n$ and let $A=\bigcup R_n$. Then $\gd(A)\leq \sup\{\gd(R_n):n\in
\mathbb{N}\}+1$.
\item[ii)] Let $(R,\fm)$ be a Noetherian regular local ring.
Then $\gd(R_{\infty})\leq \dim R+1$.
\end{enumerate}
\end{lemma}

\begin{proof} i) Without loss of generality we can assume that
$d:=\sup\{\gd(R_n):n\in \mathbb{N}\}<\infty$. First, we show that
$\pd_{A}(A/\fa)\leq d$ for all finitely generated ideals $\fa$ of
$A$.  Note that $A$ is quasi-local with unique maximal ideal
$\fm_A:=\bigcup \fm_n$. In view of Lemma \ref{dir1}, $A$ is
coherent. Since $A/\fa$ is a finitely presented $A$-module, by Lemma
\ref{g} ii), $\pd_A(A/\fa)\leq d$ if and only if
$\Tor_{d+1}^A(A/\fa,A/\fm_A)=0$. Set $\fa_i:=\fa\cap R_i$, and
recall that $\fm_i=\fm_A\cap R_i$. For any $j>d$ and for any $i$ we
observe that $\Tor_j^{R_i}(R_i/ \fa_i,R_i/ \fm_i)=0$, because
$j>d\geq \gd(R_i)$. A typical element of
$\underset{i}{\varinjlim}R_i/ \fm_i$ can be expressed by
$[a_i+\fm_i]$. The assignment $[a_i+\fm_i]\mapsto [a_i]+\fm\in
A/\fm$ provides a well-define map which is in fact  a natural
isomorphism between the rings $\underset{i}{\varinjlim}R_i/ \fm_i$
and $A/\fm$. Similarly, $\underset{i}{\varinjlim}R_i/ \fa_i\cong
A/\fa$. In view of Remark \ref{car}, we see that
$$\Tor_j^{A}(A/ \fa,A/ \fm_A)\cong\underset{i}{\varinjlim}\Tor_j^{R_i}(R_i/
\fa_i,R_i/ \fm_i)=0,$$which  shows that $\pd_{A}(A/\fa)\leq d$ for
all finitely generated ideals $\fa$ of $A$.

Now, we show that $\gd(A)\leq d+1$.  In view of Lemma \ref{g} iii),
we need to prove that $\pd_{A}(A/\fb)\leq d+1$ for all ideals $\fb$
of $A$.  Lemma 2.1 iii) asserts that $A$ is $\aleph_0$-Noetherian.
Therefore, Lemma \ref{j} iii) along with Lemma \ref{g} i) yield that
\[\begin{array}{ll}
\pd_{A}(A/ \fb)&\leq\fd_{A}(A/ \fb)+1\\&\leq\sup\{\fd(A/ \fa):\fa
\emph{ is a finitely generated ideal of }A\}+1\\&\leq\sup\{\pd(A/
\fa):\fa \emph{ is a finitely generated ideal of }A\}+1\\&\leq d+1,
\\
\end{array}\]
which completes the proof.

ii) First assume that $\Char R= 0$. Then by Definition 7.2,
$R_{\infty}$ is the directed union of regular local rings each of
them are free over the preceding. So the claim follows by i). Now,
assume that $R$ is regular and $\Char R=p$.  Let $d:=\dim R$. There
exists a system of parameters $\{x_1,\cdots,x_d\}$ of $R$ such that
$\fm=(x_1,\cdots,x_{d})$. For each positive integer $n$, set
$R_n:=\{x\in R_\infty|x^{p^n}\in R\}$. As we saw in the proof of
Lemma \ref{alef0} the ring $R_{n}$ is Noetherian and local with the
maximal ideal $(x_1^{1/p^{n}},\cdots,x_{d}^{1/p^{n}})R_{n}$ and
$\dim R_n=d$, because $R_n$ is integral over $R$. Hence $R_{n}$ is
regular, since its maximal ideal can be generated by $\dim R_{n}$
elements. By using a result of Kunz \cite[Corollary 8.2.8]{BH}, one
can find that $R_m$ is flat over $R_n$  for all $n<m$. Now, i)
completes the proof.
\end{proof}

From here on it will be assumed that $\Char R=p$ and we shall seek
to give results analogue to the results of previous sections over
$R_{\infty}$. Note that $R_{\infty}$ is quasi-local. We denote its
unique maximal ideal by $\fm_{R_{\infty}}$.

\begin{lemma} \label{gen1} Let $(R,\fm)$ be a  Noetherian local
domain  of prime characteristic $p$. Let $x_1,\cdots,x_{\ell}$ be a
finite sequence of nonzero and nonunit elements of $R_{\infty}$.
\begin{enumerate}
\item[i)]
$R_{\infty}/\sum_{i=1}^{\ell}(x_i^{\infty})R_{\infty}$ has a free
resolution of countably generated free $R_{\infty}$-modules of
length bounded by $2\ell$.
\item[ii)]
$\pd_{R_{\infty}}(R_{\infty}/\sum_{i=1}^{\ell}(x_i^{\infty})R_{\infty})\leq
\ell+1$. In particular,
$\pd_{R_{\infty}}(R_{\infty}/\fm_{R_{\infty}})\leq \dim R+1$.
\end{enumerate}
\end{lemma}

\begin{proof}  i) The proof of the first claim is a repetition of the
proof of Theorem 4.6 i).

ii) Recall from  Lemma \ref{wel} vii)  that
$\fd_{R_{\infty}}(R_{\infty}/\sum_{i=1}^{\ell}(x_i^{\infty})
R_{\infty})\leq \ell$.  By Lemma \ref{alef0}, $R_{\infty}$ is
$\aleph_0$-Noetherian. Thus, Lemma \ref{j} iii) shows that
$\pd_{R_{\infty}}(R_{\infty}/\sum_{i=1}^{\ell}(x_i^{\infty})
R_{\infty})\leq \ell+1$. Let $d:=\dim R$ and assume that
$\{x_1,\cdots,x_d\}$ is a system of parameters for $R$. To conclude
the Lemma, it remains recall from Corollary \ref{local} that
$\fm_{R_{\infty}}=\sum_{i=1}^{d}(x_i^{\infty})R_{\infty}$.
\end{proof}

Now, we recall the following result of Sally and Vasconcelos.

\begin{lemma} \label{SV} (\cite[Corollary 1.5]{SV})
Prime ideals in a 2-dimensional Noetherian local ring admit a
bounded number of generators.
\end{lemma}

\begin{lemma} \label{aux} Let $(R,\fm)$ be a  Noetherian local
domain  of prime characteristic $p$. Then any prime ideal  $\fp$ of
$R_{\infty}$ is of the form $\sum_{i=1}^{\ell}(x_{i}^{\infty})$ for
some $x_1,\cdots,x_{\ell}\in R_{\infty}$. Furthermore, if $\dim
R<3$, $\ell$ can be chosen such that it does not  depend on the
choice of $\fp$.
\end{lemma}

\begin{proof} If $\fp=0$, there is nothing to prove. So
let $\fp$ be nonzero. Let $\underline{x}:=x_1,\cdots,x_{\ell}$ be a
generating set for $\fp_0:=\fp\cap R$. When $\dim R<3$ by Lemma
\ref{SV}, $\ell$ can be chosen such that it does not  depend on the
choice of $\fp$. Clearly, we have
$$\rad((x_1^{\infty})+\cdots+(x_{\ell}^{\infty}))\subseteq\fp. \  \  (\ast)$$
By Lemma \ref{wel} iii), it is enough to show that
$\rad((x_1^{\infty})+\cdots+(x_{\ell}^{\infty}))=\fp$. For each
positive integer $n$, set $R_n:=\{x\in R_\infty|x^{p^n}\in R\}$ and
$\fa_n:=\rad((x_1^{\infty})+\cdots+(x_{\ell}^{\infty}))\cap R_n$.
Recall that $R_\infty=\bigcup R_n$. It remains to recall $\fp\cap R_n\subseteq \fa_n$ and this completes the proof.
\end{proof}

Now, we are ready to prove Theorem 1.2.

\begin{theorem}
Let $(R,\fm)$ be a  Noetherian  local domain  of prime
characteristic.
\begin{enumerate}
\item[i)] Any prime ideal of $R_{\infty}$ has a bounded free resolution of countably generated
free $R_{\infty}$-modules.
\item[ii)] If $\dim R<3$, then $\sup\{\pd_{R_{\infty}}(R_{\infty}/\fp):
\fp\in\Spec(R_{\infty})\}<\infty$.
\item[iii)]
If $R_{\infty}$ is coherent (this holds if $R$ is regular), then
$\gd(R_{\infty})\leq\dim R+1$.
\item[iv)]   If $R$ is regular and of dimension one, then $\gd(R_{\infty})=2$.
\end{enumerate}
\end{theorem}

\begin{proof}  We prove $i)$ and $ii)$ at the same time.
Let $\fp$ be a nonzero prime ideal of $R_\infty$. By Lemma
\ref{aux}, there exists a finite sequence $x_1,\cdots,x_{\ell}$ of
elements of $R_{\infty}$ such that
$\fp=\sum_{i=1}^{\ell}(x_i^{\infty})$. Also, when $\dim R<3$, the
integer $\ell$ can be chosen such that it does not depend on the
choice of $\fp$. Putting this along with Lemma \ref{gen1} i), we get
that $\pd_{R_{\infty}}(R_{\infty}/\fp)\leq 2\ell$.

iii) By Lemma \ref{alef0} $R_{\infty}$ is $\aleph_0$-Noetherian.
Also, $R_{\infty}$ is coherent. Incorporate these observations along
with the proof of Lemma \ref{dir} i) to see that
$\gd(R_{\infty})\leq\dim R+1$.

iv)   Due to the proof of Lemma \ref{dir} ii), we know  that
$R_{\infty}$ is a directed  union of discrete valuation domains, and
so $R_{\infty}$ is a $\aleph_0$-Noetherian valuation ring. By Lemma
7.1, $R_{\infty}$ is  not Noetherian. The only sticky point is to
recall from Lemma \ref{j} i) that $\pd \fa = n+1$ exactly if $\fa$
can be generated by $\aleph_n$ elements, where $\fa$ is an ideal of
$R_{\infty}$.
\end{proof}

\begin{example} i) Let $\mathbb{F}$ be a perfect field of characteristic $2$.
Consider the ring $R:=\mathbb{F}[[x^2,x^3]]$. This ring is not
regular. Note that $x\in R^+$ and $x^2\in R$. Thus $x\in R_1:=\{r\in
R^+:r^2\in R\}$. In particular, $R\subseteq\mathbb{F}[[x]]\subseteq
R_1$, and so $R_{\infty}\subseteq
(\mathbb{F}[[x]])_{\infty}\subseteq R_{\infty}$. Thus, in view of
Theorem 7.10 iv), $R_{\infty}$ is coherent and
$\gd(R_{\infty})=2<\infty$.

ii) Let $A$ be  a domain of finite global dimension on prime
spectrum. If any quadratic polynomial with  coefficient in $A$ has a
root in $A$, then $A$ is of finite global dimension on radical
ideals. Indeed, let $\fa:=\bigcap_{i\in I}\fp$ be a radical ideal of
$A$, where $\{\fp_i:i\in I\}$ is a family of prime ideals of $A$. By
\cite[Theorem 9.2]{HH1}, we know that $P:=\sum _{i\in I}\fp_i$ is
either a prime ideal or it is equal to $A$. Therefore, the desired
claim follows from the following short exact sequence of $A$-modules
$$0\longrightarrow A/ \fa\longrightarrow \bigoplus_{i\in I}
A/\fp_i\longrightarrow A/ P \longrightarrow 0.
$$
\end{example}

\begin{definition}
Let $(R,\fm)$ be a  Noetherian  local domain  of prime
characteristic $p$ and let $M$ be an $R_{\infty}$-module.
\begin{enumerate}
\item[i)] Let
$v_{\infty}:R_{\infty}\longrightarrow\mathbb{Q}\bigcup\{\infty\}$ be
the restriction of
$v:R^{+}\longrightarrow\mathbb{Q}\bigcup\{\infty\}$ to $R_{\infty}$.
This  map  provides a torsion theory for $R_{\infty}$. We denote it
by $\mathfrak{T}^{\infty}_v$.
\item[ii)]Let $x$ be an element of $R_{\infty}$. The
$R_{\infty}$-module $M$ is called almost zero with respect to $x$,
if $x^{1/p^n} M=0$ for all $n\in \mathbb{N}$. The notion
$\mathfrak{T}^{\infty}_x$  stands for the torsion theory of almost
zero $R_{\infty}$-modules with respect to $x$.
\item[iii)] Let $\fm_{R_{\infty}}$ be the unique maximal ideal of $R_{\infty}$.
One has $\fm_{R_{\infty}}^2=\fm_{R_{\infty}}$. The
$R_{\infty}$-module $M$ is called almost zero with respect to
$\fm_{R_{\infty}}$, if $\fm_{R_{\infty}}M=0$. The notion
$\mathcal{GR}^{\infty}$ stands for the torsion theory of almost zero
$R_{\infty}$-modules with respect to $\fm_{R^{\infty}}$.
\end{enumerate}
\end{definition}

As we mentioned before,  any torsion theory determines a functor. In
the next result we examine the cohomological dimension and
cohomological depth of functors induced by torsion theories of
Definition 7.12.

\begin{corollary}
Let $(R,\fm)$ be a  Noetherian  local domain  of prime
characteristic $p$ and let $x$ be a nonzero element of
$\fm_{R_{\infty}}$.
\begin{enumerate}
\item[i)] $\cd(\mathfrak{T}^{\infty}_x)=2$.
\item[ii)] $\cd(\mathcal{GR}_{\infty})\leq  \dim R+1$.
\item[iii)] If $R$ is regular, then $\cd(\mathfrak{T}^{\infty}_v)\leq \dim
R+1$.
\item[iv)]Let $M$ be an $R_{\infty}$-module with the property that
$\Ext^i_{R_{\infty}}(R_{\infty}/ \fm_{R^{\infty}},M)\neq 0$ for some
$i$. Then
$$\inf\{i\in\mathbb{N}_0:\textbf{R}^i\mathfrak{T}^{\infty}_v
(M)\neq0\}\in\{0,1,2\}.$$
\item[v)] Let $\epsilon$ be a positive nonrational real number and
let $\fa^\infty_{\epsilon}:=\{x\in R_\infty| v(x)> \epsilon\}$. Then
$\Ass_{R_\infty}(R_\infty / \fa^\infty_{\epsilon})=\emptyset$.
\end{enumerate}
\end{corollary}

\begin{proof} i) and ii) are immediately follows by Lemma \ref{gen1},
because
$$\cd(\mathfrak{T}^{\infty}_x)=\pd_{R_{\infty}}(R_{\infty}/(x^{\infty})R_{\infty})=2$$ and
$$\cd(\mathcal{GR}_{\infty})=\pd_{R_{\infty}}(R_{\infty}/\fm_{R_{\infty}}).$$

iii)  is trivial by Theorem 1.2 iv).

iv)  follows by repeating the proof of Proposition 6.5.

v) follows by repeating the proof of Proposition 6.8 i). \end{proof}

\section{Application: Homological impossibilities over
non-Noetherian rings}

In this section we list some  homological properties of Noetherian
local rings which are not hold over general commutative rings. Our
main tool for doing this is Theorem \ref{main}. Recall that a ring
is called regular, if each of its finitely generated ideal has
finite projective dimension. For example,  valuation domains are
regular. Over a Noetherian local ring $(R,\fm)$, one  can check that
$\injdim_{R} (R/\fm)=\pd_{R} (R/\fm)$. By the following, this not be
the case in the context of non-Noetherian rings, even if the ring is
coherent and regular.

\begin{example}
Let $(R,\fm)$ be a Noetherian complete local domain  of prime
characteristic. Let $E$ be the injective envelope of
$R^+/\fm_{R^+}$.
\begin{enumerate}
\item[i)] Let $F$ be an $R^+$-module. Then $F$ is  flat
if and only if the functor $-\otimes_{R^+}F$ is exact. This is the
case if and only if the functor
$$\Hom_{R^+}(-\otimes_{R^+}F,E)\cong \Hom_{R^+}(-,\Hom_{R^+}(F,E))$$
is exact.  Or equivalently, $\Hom_{R^+}(F,E)$ is an injective
$R^+$-module.
\item[ii)]Recall from Lemma \ref{wel} vii) that $R^+/\fm_{R^+}$ has a flat
resolution $\textbf{F}_\bullet$ of length bounded by $\dim R$. In
view of i) and by an application of the functor  $\Hom_{R^+}(-,E)$
to $\textbf{F}_\bullet$, we get to an injective resolution of
$\Hom_{R^+}(R^+/\fm_{R^+},E).$ Note that $R^+/\fm_{R^+}$ is a direct
summand  of $\Hom_{R^+}(R^+/\fm_{R^+},E)$. So, $\injdim_{R^+}
(R^+/\fm_{R^+})\leq
\injdim_{R^+}(\Hom_{R^+}(R^+/\fm_{R^+},E))\leq\dim R$.

\item[iii)] Assume that  $\dim R=1$. We show that $\injdim_{R^+}
(R^+/\fm_{R^+})=1$ and $\pd_{R^+} (R^+/\fm_{R^+})=2$. To see this,
in light of Theorem \ref{main} and parts i) and ii), it is enough to
prove that $R^+/\fm_{R^+}$ is not an injective $R^+$-module. But, as
$\Hom_{R^+}(R^+/\fm_{R^+},E)\cong R^+/\fm_{R^+}$ is not a flat
$R^+$-module, it turns out that $R^+/\fm_{R^+}$ is not  an injective
$R^+$-module.
\end{enumerate}
\end{example}

\begin{lemma}\label{4.2.2}
Let $R$ be a  domain and let $\fa$ be a radical ideal of $R^+$. Then
$\fa=\fa^n$ for all $n\in\mathbb{N}$.
\end{lemma}

\begin{proof}
It is enough to show that $\fa=\fa^2$. Let $x\in\fa$. Then the
polynomial $f(X)=X^2- x\in R^+[X]$ has a root $\zeta\in R^+$. As
$\fa$ is radical, from $\zeta^2=x\in\fa$, we deduce that $\zeta\in
\fa$. So $x=\zeta. \zeta\in\fa^2$.
\end{proof}

The following  gives some results on  local cohomology modules in
the context of non-Noetherian rings.

\begin{example}
\begin{enumerate}
\item[i)] Recall that for a Noetherian  ring $A$, Grothendieck's
Vanishing Theorem \cite[Theorem 3.5.7(a)]{BH} asserts that
$H^{i}_{\fa}(-):=\underset{n}{\varinjlim}\Ext^{i}_{A} (A/\fa^{n},
-)=0$ for all $i>\dim A$ and all ideals $\fa$ of $A$.  Now, let
$(R,\fm)$ be a one-dimensional Noetherian complete local domain of
prime characteristic. By Lemma \ref{4.2.2}, we observe that
$\fm_{R^+}^{n}=\fm_{R^+}$ for all $n\in\mathbb{N}$. We incorporate
this observation with Theorem \ref{main}, to see that
$$H^{2}_{\fm_{R^+}}(-)=\underset{n}{\varinjlim}\Ext^{2}_{R^+}
(R^+/\fm_{R^+}^{n}, -)=\Ext^{2}_{R^+} (R^+/\fm_{R^+}, -)\neq0.$$
Therefore, Grothendieck's Vanishing Theorem does not true for
non-Noetherian rings, even if the ring is coherent and regular.
\item[ii)]  Let $(A,\fm)$ be a Noetherian local ring  and let $M$ be
a finitely generated $A$-module. Grothendieck's non Vanishing
Theorem \cite[Theorem 3.5.7(b)]{BH} asserts that $H^{\dim
M}_{\fm}(M)\neq 0$. Let $(V,\fm_V)$ be a valuation domain with the
value group $\mathbb{Z}\bigoplus\mathbb{Z}$ such that its maximal
ideal $\fm_V=vV$ is principal. Such a ring exists, see \cite[Page
79]{M}. Recall from \cite[Definition 2.3]{Sch} that a sequence
$\underline{x}$ is called weak proregular if
$H^i_{(\underline{x})}(-)\cong H^{i}_{\underline{x}}(-)$ for all
$i$, where $H^{i}_{\underline{x}}(M)$ denotes  the $i$-th cohomology
module of $\check{C}ech$ complex of $M$ with respect to
$\underline{x}$.   Keep in mind that over integral domains any
nonzero element $x$ is weak proregular. Note that $V$ is
non-Noetherian, since its value group  does not isomorphic with
$\mathbb{Z}$. Also, recall that a ring is Noetherian if and only if
each of its prime ideals are finitely generated. Thus $d:=\dim
V\ge2$, and so $H^d_{\fm_V}(V)\cong H^{d}_{\underline{v}}(V)=0$.
Therefore, Grothendieck's non Vanishing Theorem  does not true for
non-Noetherian rings, even if the ring is coherent and regular.

\item[iii)] Let $(R,\fm)$ be a $1$-dimensional Noetherian complete
 local domain of prime characteristic  and let $v$ be a non-zero
 element of $\fm_{R^+}$. Clearly,  $H^{2}_{\underline{v}}(-)=0$. Recall
 that $H^i_{(v)}(-)\cong H^{i}_{\underline{v}}(-)$ for all $i$.
 Thus $H^2_{(v)}(-)=0$. Clearly, $\rad((v)R^+)=\fm_{R^+}$. By
 part ii), $H^2_{\fm_{R^+}}(-)\neq0$. Therefore,  $H^{i}_{\fa}(-)$
  and $H^{i}_{\rad(\fa)}(-)$ are not necessarily isomorphic, where
   $\fa$ is an ideal of a general commutative ring.
\end{enumerate}
\end{example}

Next, we exploit Example 8.2 ii) to show that the Intersection
Theorem is not true for a general commutative ring, even if the ring
is coherent and regular.

\begin{example} Let $(R,\fm)$ be a Noetherian local ring and
$M, N \neq0$ are finitely generated $R$-modules such that
$M\otimes_R N$ has finite length. Recall from \cite[Section 9.4]{BH}
that  the Intersection Theorem asserts  that $\dim N \leq \pd M$.
Now, let $(V,\fm_V)$ be the valuation domain of Example 8.2 ii).
Then $\dim V\ge2$. Now, consider the $V$-modules $M:=V/ \fm_V$ and
$N:=V$. Clearly, $\ell(M\otimes_V N)=1$. But $\dim N\ge2$ and $\pd
M=1$. Thus Intersection Theorem is not true for coherent and regular
rings.
\end{example}

\section{Some questions}

Theorem \ref{main} gives no information when $\mathbb{Q}\subseteq
R$. So,  we ask the following:

\begin{question}
Let $(R,\fm)$ be a Noetherian complete local domain of
equicharacteristic zero. Is $\pd_{R^+} (R^+/\fm_{R^+})<\infty$?
\end{question}

Theorem \ref{perfect} gives some partial answers to:
\begin{question}
Let $(R,\fm)$ be a Noetherian local domain of prime characteristic.
Let $\fa$ be an ideal of $R_{\infty}$. Under what conditions
$\pd_{R_{\infty}} (R_{\infty}/\fa)< \infty$?
\end{question}

Proposition 6.4 gives no information when $\dim R\geq2$. So, we ask
the following:

\begin{question}
Let $(R,\fm)$ be a Noetherian complete local domain of dimension
greater than one. Is $\cd(\mathfrak{T}_v)$ bounded?
\end{question}

Let $A$ be a ring and $\mathfrak{T}$  a torsion theory of
$A$-modules.  We say that an $A$-module $M$  has finite almost
(flat) projective dimension with respect to $\mathfrak{T}$ if there
exists the following complex of (flat) projective $R^+$-modules
$$\textbf{P}_\bullet:0\lo P_n\lo\cdots \lo P_1\lo P_0\lo 0$$ such that
$H_0(\textbf{P}_\bullet)=M$ and
$H_i(\textbf{P}_\bullet)\in\mathfrak{T}$ for all $i>0$. We say that
a ring $A$ is almost regular with respect to $\mathfrak{T}$ if any
module has finite almost projective dimension with respect to
$\mathfrak{T}$.

\begin{example}
Let $(R,\fm)$ be a three-dimensional complete local domain of mixed
characteristic. Recall that $H_{\underline{m}}^{i}(R^+)$ is the
$i$-th cohomology module of $\check{C}ech$ complex of $R^+$ with
respect to a generating set of $\underline{\fm}$. By Lemma \ref{wel}
v), $R^{+}=\bigcup R'$, where the rings $R'$ are Noetherian and
normal. Any normal ring satisfies the Serre's condition
$\mathbb{S}_2$. Hence, any length two system of parameters for $R$
is a regular sequence for $R'$, and so for $R^+$. Thus, the
classical grade of $\underline{\fm} R^+$ on $R^+$ is greater than
$1$. Consequently,
$H^1_{\underline{m}}(R^+)=H^0_{\underline{m}}(R^+)=0$, because
$\check{C}ech$ grade is an upper bound for the classical grade, see
e.g. \cite[Proposition 2.3(i)]{AT1}. Now, let $x$ be a nonzero
element of $\fm$. In light of \cite[Theorem 0.2]{He2}, we see that
$x^{1/n}$ kills $H^2_{\underline{m}}(R^+)$ for arbitrarily large
$n$. By inspection of $\check{C}ech$ complex of $R^+$ with respect
to $\underline{\fm} R^+$, we see that $H^3_{\underline{m}}(R^+)$ has
finite almost flat dimension with respect to $\mathfrak{T}_x$.
\end{example}

\begin{question}
Let $(R,\fm)$ be a Noetherian local domain. Let $\mathfrak{T}_v$ be
as Definition 6.1 i) and $\mathfrak{T}_v^{\infty}$ be as Definition
7.12. i).
\begin{enumerate}
\item[i)] Assume that $R$ is complete. Is $R^+$ almost regular
with respect to $\mathfrak{T}_v$?
\item[ii)] Assume that $R$ has prime characteristic.
Is $R_{\infty}$ almost regular with respect to
$\mathfrak{T}^{\infty}_v$?
\end{enumerate}
\end{question}

\section{Appendix}

In this appendix we show that  bounds of Theorem 1.2 are sharp.

\begin{lemma} \label{aux} Let $(R,\fm)$ be a $d$-dimensional Noetherian local domain  of prime characteristic $p$. Then any radical ideal of $R^{\infty}$ is of the form $\sum_{i=1}^{d}(\alpha_{i}^{\infty})$ for some $\alpha_1,\ldots,\alpha_{d}\in R^{\infty}$.
\end{lemma}

\begin{proof}
Let $\fa$ be a radical ideal of $R^{\infty}$ and set $\fb:=\fa\cap R$.  Clearly, $\fb$ is radical. In view of \cite[Theorem 9.13]{I1},  there is a finite sequence $\underline{\alpha}:=\alpha_1,\ldots,\alpha_{d}$ of elements of $R$ such that $\sqrt{\underline{\alpha}R}=\sqrt{\fb}=\fb$. Suppose $x\in \fa$. Then $x^{p^m}\in R\cap \fa=\fb$ for some integer $m$. It yields that $x^{p^n}=r_1\alpha_1+\cdots+r_d\alpha_d$ for some integer $n$ where $r_i\in R$. By taking $p^n$-th root, $x=r_1^{1/p^n}\alpha_1^{1/p^n}+\cdots+r_d^{1/p^n}\alpha_d^{1/p^n}$. So $x\in \sum_{i=1}^{d}(\alpha_{i}^{\infty})$.
\end{proof}

By  $R^{+}$, we mean the integral closure of  a domain $R$ in an algebraic closure of the fraction field of $R$.

\begin{proposition}  \label{perfect1}
Let $(R,\fm)$ be a  Noetherian  local domain  of prime characteristic. The following holds:
\begin{enumerate}
\item[$\mathrm{(i)}$] If $A$ is either $R^{\infty}$ or $R^+$, then $\sup\{\fd_{A}(A/\fa):\fa \textit{ is a radical ideal }\}\leq \dim R$.
\item[$\mathrm{(ii)}$]  We have $\sup\{\pd_{R_{\infty}}(R_{\infty}/\fa):\fa \textit{ is a radical ideal }\}\leq \dim R+1$.
\end{enumerate}
\end{proposition}

\begin{proof}
$\mathrm{(i)}$: Let $\fa$ be a radical ideal of $A$. In the case $A=R^{\infty}$ the claim follows by Lemma \ref{aux}. In the other case, recall that $R^{+}=\bigcup R_\gamma$, where $R_\gamma$ is module-finite extension of $R$ and  $R^{+}={\varinjlim}R_\gamma^\infty$.  Set $\fa_\gamma:=\fa \cap R_\gamma^\infty$.  Clearly,  $\fa_\gamma$ is radical. There is a natural isomorphism between the rings ${\varinjlim}R_\gamma^\infty/ \fa_\gamma$ and $A/\fa$. Similarly, ${\varinjlim}R_\gamma/ \fa_\gamma\cong A/\fa$. To conclude, it remains to recall from  \cite[VI, Exercise 17]{CE} that $$\underset{i}{\varinjlim}\Tor_j^{R_\gamma^\infty}(R_\gamma^\infty/ \fa_\gamma,-)\cong\Tor_j^{R^{+}}(R^{+}/ \fa,-).$$

$\mathrm{(ii)}$: In view of  Lemma \ref{aux}, the claim follows by Lemma 7.7.
\end{proof}

A ring $R$ is called F-coherent if $R^{\infty}$ is a coherent ring.
Regular rings are examples of F-coherent. Let us recall
the following from \cite{AB}.

\begin{lemma} \label{a}
Let $R$ be an $F$-coherent reduced ring. 
Assume that $R$ is either excellent or  homomorphic image of a Gorenstein local ring. Then $R^{\infty}$ is big Cohen-Macaulay.
\end{lemma}

\begin{corollary}\label{S1}
Let $R$ be a  Noetherian local F-coherent domain.  The following holds:
\begin{enumerate}
\item[$\mathrm{(i)}$]
 If $R$ is complete, then $R^{\infty}$ is the perfect closure of a normal F-finite domain.
\item[$\mathrm{(ii)}$] If $R$ is $1$-dimension, then   $R^{\infty}$ is a valuation domain.
\end{enumerate}
\end{corollary}

\begin{proof}
$\mathrm{(i)}$ By \cite[Corollary 6.2.10]{Gl}, coherent regular rings are greatest common divisor and so  integrally closed. By Theorem 1.2, $R^{\infty}$ is integrally closed.
Let $F$ be the fraction field of $R$ and set $A:=F\cap R^{\infty}$. Take $x\in A^{\infty}$. Then $x^{p^n}\in A$ for some $n$. Due to definition of $A$, $x^{p^n}\in R^{\infty}$. Thus $x^{p^{n+m}}\in R$ for some $m$. So $x\in R^{\infty}$, i.e., $A^{\infty}= R^{\infty}$. We know that $A$ is Noetherian, because $R$ is complete. This yields that $A$ is a complete local normal domain.  By replacing $R$ with $A$, we may assume that $R$ is complete and normal. It remains to apply the Gamma construction $R^{\Gamma}$ (for definition and more details on tight closure see \cite{H2}). Then $R\to R^{\Gamma}$ is  purely inseparable extension and $R^{\Gamma}$ is F-finite. Also, $R^{\Gamma}$ is domain and  is flat over $R$. Since the prefect closure of $R$ and $R^{\Gamma}$ are the same, we need to show  $R^{\Gamma}$ is normal. $R$ satisfies Serre's condition $(R_1)$  and $(S_2)$.   Fibers of $R\to R^{\Gamma}$ are Gorenstein. Thus $R^{\Gamma}$ satisfies $(S_2)$. If $I$ defines the singular locus of $R$, then $IR^{\Gamma}$ defines the singular locus of $IR^{\Gamma}$. This yields that $R^{\Gamma}$ satisfies $(R_1)$, since $R$ satisfies $(R_1)$. By applying Serre characterization of normality, $R^{\Gamma}$ is normal.

$\mathrm{(ii)}$ First note that normalization of $R$ is Noetherian, since $R$ is $1$-dimension. By applying the proof of part $(i)$, we can assume that $R$ is normal. Hence, $R$ is discrete valuation domain. For each positive integer $n$, set $R_n:=\{x\in R^\infty|x^{p^n}\in R\}$, where $p$ is characteristic of $R$. One may find easily that $R_n$ is discrete valuation domain. Direct union of tower of valuation domain is a valuation domain. This along with $R^\infty=\bigcup R_n$,  yields the claim.
\end{proof}

\begin{lemma}  \label{v}(see \cite[Theorem 6.2.15]{Gl})
Let $(A,\fm_A)$ be a quasilocal coherent ring such that $\Wdim(A)=\gd(A)<\infty$. Then $\fm_A$ is finitely generated.
\end{lemma}

\begin{lemma}  \label{sy}
Let $A$ be a  coherent ring and $M$ a finitely presented $A$-module of finite projective dimension. Then
$\pd(M)=\fd(M)$.
\end{lemma}

\begin{proof}
First note that $n:=\fd(M)<\infty$. We prove the claim by induction on $n$. In the case $n=0$, there is no thing to prove, since finitely presented flat modules are projective. Due to the coherent assumption, we know that syzygy modules are finitely presented. To prove the claim in general case, it remains deal with the syzygies.
\end{proof}

Let $\fa$ be an ideal of a ring $A$ and $M$ an $A$-module. Denote the classical grade  of $\fa$ on $M$,  by $\cgrade_A(\fa,M)$. The polynomial grade of $\fa$ on $M$ is defined by $$\pgrade_A(\fa,M):=\underset{m\rightarrow\infty}{\lim}\cgrade_{A[t_1, \ldots,t_m]}(\fa A[t_1, \ldots,t_m],A[t_1,\ldots,t_m]\otimes_A M).$$
Recall that if $(A,\fm_A)$ is  quasilocal  and $M$ has a finite free resolution, then $$\pd_A(M) +\pgrade_A(\fm_A,M) = \pgrade_A(\fm_A,A).$$

\begin{theorem}
Let $R$ be a  Noetherian local F-coherent domain which is either excellent or homomorphic image of a Gorenstein local ring.  If $R$ is not a field, then $\gd(R^{\infty})=\dim R+1$.
\end{theorem}

\begin{proof}  Let $\underline{x}:=x_1,\ldots,x_d$ be a system of parameters for $R$. By Lemma \ref{a}, $\underline{x}$ is regular sequence on $R^{\infty}$. By Auslander-Buchsbaum formula  and Lemma  \ref{sy}, $\fd(R^{\infty}/ \underline{x}R^{\infty})=d.$ Combining this with  Lemma \ref{a}, $\Wdim(R^{\infty})= \dim R$. Now suppose on the contrary that $\gd(R^{\infty})\neq\dim R+1$. It turns out that $\gd(R^{\infty})=\dim R$. Note that $R^{\infty}$ is quasilocal. Call its maximal ideal by $\fm_{R^{\infty}}$. In light of Lemma~\ref{v}, we see that  $\fm_{R^{\infty}}$ is finitely generated. One may find easily that $\fm_{R^{\infty}}=\fm_{R^{\infty}}^2$.  By Nakayama's Lemma, $\fm_{R^{\infty}}=0$. This is the case if and only if $R$ is a field. This contradiction shows that  $\gd(R^{\infty})=\dim R+1$.
\end{proof}

\begin{example}\label{elg}In this exanple we study global dimension of certain perfect domains.
\begin{enumerate}
\item[$\mathrm{(i)}$]
Let $F$ be an algebraically closed field of prime characteristic and let $(A,\fm_A)$ be the localization of $F[X_1,X_2,\ldots]$ at $(X_1,X_2,\ldots)$. Set $(R, \fm_R)$ be the perfect closure of $A$. Here we show that $\fd_R(R/ \fm_R)=\infty$. To this end, let $(A_i,\fm_i)$ be the localization of $F[X_1,\ldots,X_i]$ at $(X_1,\ldots,X_i)$. By applying \cite[Corollary 8.2.8]{BH}, $A_i^\infty$ is flat over  $A_i$. This yields that  ${\varinjlim}A_i^\infty$ is flat over  ${\varinjlim}A_i$, i.e.,  $R$ is flat over $A$. Hence, $$\Tor_i^{A}(A/ \fm_A,A/\fm_A)\otimes_AR\cong\Tor_i^{R}(R/ \fm_R,R/\fm_R).$$ Thus, to conclude its enough to show that $\Tor_j^{A}(A/ \fm,A/\fm)\neq0$ for all $j$. Let $i$ be a non negative integer. Then, $\Tor_i^{A_i}(A_i/ \fm_i,A_i/\fm_i)\neq0$. By using rigidity of Tor's modules, we have $\Tor_j^{A_i}(A_i/ \fm_i,A_i/ \fm_i)\neq0$ for any $j\leq i$. So $A_i/ \fm_i$ can be embedded in $\Tor_j^{A_i}(A_i/\fm_i,A_i/ \fm_i)$. One can prove that $\underset{i}{\varinjlim}A_i/\fm_i\cong A/\fm$. Note that $$0\neq\underset{i}{\varinjlim}A_i/\fm_i\hookrightarrow\underset{i}{\varinjlim}\Tor_j^{A_i}(A_i/\fm_i,A_i/ \fm_i)\cong\Tor_j^{A}(A/ \fm,A/\fm),$$
which yields the claim.
\item[$\mathrm{(ii)}$]  Recall that torsion free modules over valuation domains are flat. It turns out that Frobenius map over a valuation domain of prime characteristic is flat. Now, let $V$ be a valuation of prime characteristic with infinite global dimension.
\item[$\mathrm{(iii)}$]
The regularity of $R$ in the main result is really needed. Consider the 1-dimensional ring
$R:=\mathbb{F}_3[x,y]/(y^2-x^3-x^2)$. One can shows that
$\gd(R_{\infty})>2$ (see \cite[Example 7.7]{AS}).
\end{enumerate}
\end{example}

\begin{acknowledgement}  I would like to thank  P. Roberts  for valuable suggestions  on the earlier
version of this paper. I also thank M. Tousi and K. Divaani-Aazar
for their  help as well as the referee for suggestions improving the
presentation of the results.
\end{acknowledgement}


\end{document}